\documentclass[11pt]{amsart}

\usepackage{amssymb,amsmath,amscd,amsfonts}
\usepackage{mathtools}
\usepackage{graphicx}
\usepackage{xcolor}
\usepackage{mathrsfs}
\usepackage{stmaryrd}
\usepackage{epsfig,color}
\usepackage{blindtext}
\usepackage{enumerate}
\usepackage{braket}
\usepackage{comment}

\usepackage{geometry}
\usepackage[colorlinks,linktocpage]{hyperref}
\hypersetup{linkcolor=[rgb]{0,0,0.715}}
\hypersetup{citecolor=[rgb]{0,0.715,0}}

\newtheorem{thm}{Theorem}[section]
\newtheorem{cor}[thm]{Corollary}
\newtheorem{prop}[thm]{Proposition}
\newtheorem{lem}[thm]{Lemma}

\newtheorem{mainthm}{Theorem}

\theoremstyle{definition}
\newtheorem{defn}[thm]{Definition}

\newtheorem{notn}[thm]{Notation}

\newtheorem{rem}[thm]{Remark}
\newtheorem{rems}[thm]{Remarks}

\let\phi\varphi

\newcommand{\dotprod}[2]{\ensuremath{\langle #1 , #2 \rangle}}

\newcommand{\R}{\mathbb{R}}

\renewcommand{\subset}{\subseteq}

\newcommand{\dvol}{\mathrm{dvol}}
\newcommand{\Ric}{\mathrm{Ric}}
\newcommand{\dist}{d}

\newcommand{\meas}{\mathfrak{m}}
\newcommand{\di}{\mathop{}\!\mathrm{d}}
\DeclareMathOperator{\RCD}{RCD}
\DeclareMathOperator{\CDe}{CD}

\newcommand{\F}{{\mathcal{F}}}
\newcommand{\D}{{\mathcal{D}}}
\newcommand{\Le}{{\mathcal{L}}}

\makeatletter
\let\c@equation\c@thm
\makeatother
\numberwithin{equation}{section}

\begin{document}

\title[Universal non-CD of sub-Riemannian manifolds]{Universal non-CD of sub-Riemannian manifolds}

\author{Dimitri Navarro}
\address[Dimitri Navarro]{Department of Mathematics, University of California, Santa Cruz, CA, USA.}
\email{dnavar17@ucsc.edu}

\author{Jiayin Pan}
\address[Jiayin Pan]{Department of Mathematics, University of California, Santa Cruz, CA, USA.}
\email{jpan53@ucsc.edu}

\maketitle

\begin{abstract}
   We prove that a sub-Riemannian manifold equipped with a full-support Radon measure is never $\CDe(K,N)$ for any $K\in \mathbb{R}$ and $N\in (1,\infty)$ unless it is Riemannian. This generalizes previous non-$\CDe$ results for sub-Riemannian manifolds \cite{Juillet_20,Magnabosco_Rossi_AlmostRiem_23,Rizzi-Stefano_2023}, where a measure with smooth and positive density is considered. Our proof is based on the analysis of the tangent cones and the geodesics within. Secondly, we construct new $\RCD$ structures on $\mathbb{R}^n$, named cone-Grushin spaces, that fail to be sub-Riemannian due to the lack of a scalar product along a curve, yet exhibit characteristic features of sub-Riemannian geometry, such as horizontal directions, large Hausdorff dimension, and inhomogeneous metric dilations.
\end{abstract}

\tableofcontents

\parskip=4pt plus 0.5pt

\section{Introduction}\label{sec:intro}

\subsection{Main results}

In their pioneering papers, Lott--Villani  \cite{Lott_Villani_2009} and Sturm \cite{Sturm_I_06,Sturm_II_06} independently introduced a synthetic notion of Ricci curvature lower bounds on metric measure spaces by using the theory of optimal transport. Their combined work led to the theory of $\CDe(K,N)$ spaces, where $K$ and $N$ represent a lower bound on the Ricci curvature and an upper bound on the dimension, respectively. For a complete Riemannian manifold $(M,g)$, the metric measure space $(M,d_g,\mathrm{vol}_g)$ satisfies the $\CDe(K,N)$ condition if and only if $\mathrm{Ric}\ge K$ and $\dim M\le N$. Moreover, the class of $\CDe(K,N)$ spaces is stable under measured Gromov--Hausdorff convergence and thus includes Ricci limit spaces, which were first systematically studied by Cheeger and Colding in their seminal works \cite{CCI,CCII,CCIII}. We refer the readers to \cite{Villani_09} for an introduction to $\CDe(K,N)$ spaces.

An alternative framework extending the Riemannian geometry is sub-Riemannian geometry. On a smooth manifold $M$, a sub-Riemannian structure is given by a (possibly rank-varying) distribution $\mathcal{D}\subseteq TM$ that satisfies the H\"ormander condition and is equipped with a smoothly varying scalar product. Thanks to Chow--Rashevskii theorem, the sub-Riemannian structure induces the Carnot-Carathéodory distance $d_{sR}$ on $M$. When $\mathcal{D}=TM$, one recovers the Riemannian case. We refer the readers to \cite{ABB,Bellaiche_96,Montgomery_book} for background on sub-Riemannian geometry.

A natural question in the field is whether a sub-Riemannian manifold, equipped with a measure, can satisfy curvature-dimension conditions. A series of important works have explored the non-$\CDe$ property for sub-Riemannian manifolds; notably, the works by Juillet \cite{Juillet_20}, Magnabosco--Rossi \cite{Magnabosco_Rossi_AlmostRiem_23}, and Rizzi--Stefani \cite{Rizzi-Stefano_2023}. In these works, the sub-Riemannian manifold is equipped with a measure with a smooth and positive density (in each local chart). The most general result to date is due to Rizzi--Stefani \cite{Rizzi-Stefano_2023}: if $M$ is a sub-Riemannian manifold such that $\mathcal{D}\neq TM$, then for any smooth and positive measure $\meas$ on $M$, the metric measure space $(M,d_{sR},\meas)$ does not satisfy the $\CDe(K,N)$ condition for any $K\in \mathbb{R}$ and $N\in (1,\infty]$.

A different picture emerges if we allow a sub-Riemannian manifold to have boundary points and equip it with a measure that degenerates on the boundary. We recall that the Grushin plane, a classical example in sub-Riemannian geometry \cite[Section 3.1]{Bellaiche_96}, is $\mathbb{R}^2$ with a sub-Riemannian structure whose distribution is generated by the vector fields $\partial_x$ and $x\partial_y$; equivalently, it is the metric completion of the Riemannian metric $dx^2 + x^{-2} dy^2$ defined on $\{x\not= 0\}$. In \cite{Pan-Wei_2022}, Pan and Wei showed that half of the Grushin plane $\{x\ge 0\}$, equipped with a weighted measure
$\meas=x^p dxdy$, is a Ricci limit space for sufficiently large $p$; also see \cite[Section 3]{DHPW_2023} and \cite{Pan_2023} for related constructions and properties. Alternatively, one can verify that this weighted Grushin halfplane is $\CDe(0,N)$ by computing the $N$-Barky--{\'E}mery curvature and showing that the open halfplane $\{x>0\}$ is geodesically convex; see \cite[Section 3.5]{Rizzi-Stefano_2023}. 

We observe that this example of weighted Grushin halfplane does not contradict the above-mentioned non-$\CDe$ result for sub-Riemannian manifolds for two reasons. First, the Grushin halfplane has a boundary; thus, it is not a manifold (without boundary). Secondly, the measure $\meas=x^p dxdy$ degenerates on the boundary $\{x=0\}$, so $\meas$ does not have a positive density. This exposes a gap in our understanding. The existing non-$\CDe$ results apply to sub-Riemannian manifolds (without boundary) but require a smooth positive measure. The existing $\CDe$ example has a degenerated measure but also has a boundary. This leads directly to our main question: \textit{What happens on a sub-Riemannian manifold (without boundary) if we remove the positivity assumption on the measure? Specifically, does the non-$\CDe$ property hold for a sub-Riemannian manifold equipped with an arbitrary Radon measure?}.

As the main result of this paper, we resolve this question by establishing that the non-$\CDe$ property is a universal feature of sub-Riemannian manifolds, regardless of the choice of reference measure, provided it has full support.

\begin{mainthm}\label{mainthm:not_CD}
 Let $(M,d_{sR})$ be a sub-Riemannian manifold such that $\mathcal{D}\neq TM$. Then for any full-support Radon measure $\meas$, the metric measure space  $(M,d_{sR},\meas)$ is not $\CDe(K,N)$ for any $K\in \mathbb{R}$ and $N\in (1,\infty)$. 
\end{mainthm}

\begin{rems}
Let us give some remarks on Theorem \ref{mainthm:not_CD}.\\
(1) We clarify that the manifold $M$ in Theorem \ref{mainthm:not_CD} is a manifold without boundary; otherwise, the non-$\CDe$ property no longer holds as we have seen in the weighted Grushin halfplane.\\
(2) The full-support condition on $\mathfrak{m}$ is required. To see this, let us consider the Grushin (full)plane $(\mathbb{R}^2,d_{sR})$ equipped with the measure $${\meas}=m(x) dxdy, \text{ where } m(x)=x^p \cdot \chi_{[0,\infty)}$$
and $\chi_{[0,\infty)}$ denotes the characteristic function of $[0,\infty)$. This measure $\meas$ does not have full-support on $\mathbb{R}^2$. The resulting metric measure space $(\mathbb{R}^2,d_{sR},\meas)$ is $\CDe(0,N)$ for suitable $p$ and $N$.\\
(3) The infinite-dimensional case $N=\infty$ is unclear to us. Our approach to Theorem \ref{mainthm:not_CD} relies on Gigli's splitting theorem for $\RCD(0,N)$ spaces \cite{Gigli13,Gigli14}, which fails when $N=\infty$ as pointed out in \cite[page 6]{Gigli13}.
\end{rems}

Beyond sub-Riemannian manifolds, the non-$\CDe$ property also holds for some particular classes of
sub-Finsler manifolds with smooth positive measures; see  \cite{BMRT_24,Magnabosco_Rossi_sub-Finsler_23,Magnabosco_Rossi_Review_2025} and references therein. Our method does not extend to sub-Finsler manifolds because it depends on the infinitesimal Hilbertian property of sub-Riemannian manifolds \cite{Le_Donne_23}. As a corollary of Theorem \ref{mainthm:not_CD}, we obtain the universal non-$\RCD$ property below for sub-Finsler manifolds.

\begin{cor}\label{cor: subFinsler_not_RCD}
  Let $(M,d_{sF})$ be a sub-Finsler manifold. If there exists a full-support Radon measure $\meas$ such that $(M,d_{sF},\meas)$ satisfies the $\RCD(K,N)$ condition for some $K\in \mathbb{R}$ and $N\in (1,\infty)$, then $(M,d_{sF})$ is Riemannian.
\end{cor}

Our main theorem demonstrates a fundamental incompatibility between sub-Riemannian structures and the $\CDe$ condition on manifolds regardless of the reference measure. This naturally raises the question: how ``close" can a manifold be to sub-Riemannian while still satisfying the $\CDe$ or $\RCD$ condition? As a secondary result, we construct new $\RCD$ structures on $\mathbb{R}^n$ that fail to be sub-Riemannian due to the lack of scalar product along a one-dimensional curve in $\mathbb{R}^n$, yet exhibit characteristic features of sub-Riemannian geometry, such as horizontal directions, inhomogeneous metric dilations, and large Hausdorff dimension. 

These examples extend the constructions of Pan and Wei \cite{Pan-Wei_2022, Pan_2023}, which were limited to manifolds with boundary (halfplanes or hemispheres). Our construction, to our knowledge, provides the first examples of $\RCD$ structures on manifolds without boundary where the Hausdorff dimension strictly exceeds the topological dimension.

\begin{mainthm}\label{mainthm:exmp}
    Given any $n\ge 4$ and $\alpha>0$, there is an $\RCD(0,N)$ structure $(\mathbb{R}^n,d,\meas)$, where $N$ is sufficiently large, such that:\\
    (1) $d$ is the metric completion of a smooth Riemannian metric $g$ on $\mathbb{R}^n-C$, where 
    $$C=\{(0,y)\in \mathbb{R}^{n-1}\times \R\ |\ y\in \mathbb{R}\}$$ is a coordinate axis in $\mathbb{R}^n$, but $d$ does not come from any sub-Riemannian metric on $\mathbb{R}^n$.\\
    (2) The set of all horizontal directions
    $$\Delta:=\bigcup_{x\in \mathbb{R}^n} \{ \gamma'(0)\in T_x \R^n |\  \gamma \text{ is a $C^1$-curve of finite length with $\gamma(0)=x$} \} $$
    coincides with a (rank-varying) proper distribution generated by a finite family of smooth vector fields that satisfies the H\"ormander condition.\\
    (3) The metric space $(\mathbb{R}^n,d)$ admits a family of metric dilations; specifically, for each $\lambda>0$, the map
    $$\delta_\lambda: \mathbb{R}^n\to \mathbb{R}^n,\quad (x,y)\to (\lambda x, \lambda^{1+\alpha} y),$$
    where $x\in\R^{n-1}$ and $y\in \R$, satisfies
    $$d(\delta_\lambda(x,y),\delta_\lambda(x',y'))=\lambda\cdot d((x,y),(x',y'))$$
    for all $(x,y),(x',y')\in \R^n$.\\
    (4) The singular curve $C$ has Hausdorff dimension $1+\alpha$; in particular, when $\alpha>n-1$, $(\mathbb{R}^n,d)$ has Hausdorff dimension $1+\alpha>n$.
\end{mainthm}

By Theorem \ref{mainthm:not_CD}, one cannot find a scalar product on $\Delta|_{C}$ that produces the same $d$; otherwise, we would obtain a sub-Riemannian structure that satisfies the $\CDe$ condition. In fact, around any singular point $x\in C$, the Riemannian metric $g$ is a doubly warped product of a cone metric and a Grushin metric. These $\RCD$ structures are built as Ricci limit spaces. 

To close the introduction, let us mention that, due to the failure of the $\CDe(K,N)$ condition in the sub-Riemannian setting, other synthetic curvature conditions have been explored for sub-Riemannian manifolds, for example, the measure concentration property \cite{Badreddine_Rifford_20, Barilari_Rizzi_18, BMRT_24, Borza_Rizzi_25, BMRT25, Juillet_09, Rifford_13, Rizzi_16}, sub-Riemannian Bakry-{\'E}mery curvature \cite{BG17, Barilari_Rizzi_20}, quasi curvature-dimension condition \cite{Milman_21}, and $\CDe(\beta,N)$ condition in the setting of gauge metric measure spaces \cite{Barilari_Mondino_Rizzi_22}, which unifies the Riemannian and sub-Riemannian cases, inspired by \cite{Barilari_Rizzi_19} and \cite{BKS18}.

\subsection{Proof strategy of Theorem \ref{mainthm:not_CD}}

The proof of our Theorem \ref{mainthm:not_CD} requires a new approach, as prior non-$\CDe$ results \cite{Juillet_20,Magnabosco_Rossi_AlmostRiem_23,Rizzi-Stefano_2023} fundamentally rely on the positivity of the measure. Our strategy is to shift focus from the measure to the underlying metric geometry, specifically by analyzing the structure of metric tangent cones and the behavior of geodesics within them. This approach has the advantage of being largely insensitive to the reference measure. Below are the main ingredients involved in our proof of Theorem \ref{mainthm:not_CD}:\\
$\ \ \bullet$ The tangent cone of $M$ at any point is a homothetic sub-Riemannian structure on $\mathbb{R}^n$ \cite{Bellaiche_96};\\
$\ \ \bullet$ Gigli's splitting theorem for $\RCD(0,N)$ spaces \cite{Gigli13,Gigli14};\\
$\ \ \bullet$ A regularity result on (abnormal) geodesics in sub-Riemannian manifolds under suitable blow-up by Monti, Pigati, and Vittone \cite{Monti18};\\
$\ \ \bullet$ A universal Hilbertian result by Le Donne, Lučić, and Pasqualetto stating that a sub-Riemannian manifold with any Radon measure is infinitesimally Hilbertian \cite{Le_Donne_23}.

To illustrate the main idea to prove Theorem \ref{mainthm:not_CD}, let us first assume that all minimizing geodesics are of class $C^1$. We argue by contradiction and suppose that $(M,d_{sR},\meas)$ satisfies the $\CDe(K,N)$ condition. Thanks to the universal Hilbertian property \cite{Le_Donne_23}, $(M,d_{sR},\meas)$ is an $\RCD(K,N)$ space. Let $p\in M$ be a point such that $n_1:=\dim(\mathcal{D}_p)< \dim M=:n$. We consider a tangent cone of $M$ at $p$, which is a sub-Riemannian manifold $(\mathbb{R}^{n},\hat{d})$. Equipped with a limit renormalized measure $\hat{\meas}$, the metric measure tangent cone $(\mathbb{R}^n,\hat{d},\hat{\meas})$ is $\RCD(0,N)$, thus Gigli's splitting theorem applies. The blow-up of normal geodesics emanating at $p$ provides $n_1$ many independent lines in $(\mathbb{R}^n,\hat{d})$. Hence $(\mathbb{R}^n,0,\hat{d},\hat{\meas})$ is isomorphic to a product $(\mathbb{R}^{n_1},0,d_E,\mathcal{L})\otimes (Z,z,d_Z,\meas_Z)$, where $Z$ is not a point due to the hypothesis $n_1<n$. We look for a contradiction between this metric splitting structure $\mathbb{R}^{n_1}\times Z$ and the sub-Riemannian structure $(\R^n,\hat{d})$. The contradiction arises from analyzing geodesics in $Z$. Any geodesic $\gamma$ of $Z$ starting at $z$ is also a geodesic in the sub-Riemannian manifold $(\mathbb{R}^n,\hat{d})$. Due to the sub-Riemannian structure on $\mathbb{R}^n$, the initial tangent vector of $\gamma$ should belong to the $\mathbb{R}^{n_1}$-factor, and this should lead to a contradiction since $\gamma$ is contained in the other orthogonal factor $\{0\}\times Z$. 

The primary difficulty in making this sketch rigorous is the potential lack of regularity of minimizing geodesics. In fact, it is unknown whether all minimizing geodesics in a sub-Riemannian manifold are of class $C^1$. By the Pontryagin maximum principle, a minimizing geodesic in a sub-Riemannian manifold is either \textit{normal} or \textit{abnormal}, where the latter case was first constructed by Richard Montgomery \cite{Montgomery_1994}. Although normal geodesics must be smooth, no further regularity beyond the Lipschitz one is known for abnormal geodesics. Recently, surprising examples of non-smooth abnormal geodesics have been constructed \cite{CJMRS_2025}; these minimizers are of class $C^2$ but not $C^3$. 

To circumvent this regularity issue on the geodesic $\gamma$, we apply a result of Monti--Pigati--Vittone \cite{Monti18}. Their result ensures that, even when $\gamma$ is abnormal, there is a suitable blow-up sequence such that $\gamma$ converges to an admissible curve of constant control (in particular, a smooth curve) under this blow-up. Then we can further construct a normal geodesic in $M$ that converges to a limit geodesic with the same initial tangent vector when passed to the tangent cone; moreover, this limit geodesic stays in $\mathbb{R}^{n_1}\times \{z\}$. Now we have two geodesics in $(\mathbb{R}^n,\hat{d})$ that are contained in two orthogonal factors but simultaneously also tangential at certain small scales, resulting in the desired contradiction. 

The result in \cite{Monti18} is stated for distributions of constant rank. As our setting includes rank-varying distributions, for the reader's convenience, we provide a proof of the necessary extension in Appendix \ref{sec:appendix} by modifying the argument in \cite{Monti18,Monti17}. 

\emph{Acknowledgments.}

The authors are thankful to Richard Montgomery for introducing sub-Riemannian geometry to them and for helpful discussions on sub-Riemannian geodesics. The authors would like to thank Mattia Magnabosco for discussions on sub-Finsler manifolds.

J. Pan is partially supported by the National Science Foundation DMS-2304698 and Simons Foundation Travel Support for Mathematicians. 

\section{Preliminaries}\label{sec:pre}
\subsection{Sub-Riemannian manifolds}\label{sec:prelin_sub-riem}

Let $M^n$ be a connected smooth manifold without boundary. A \emph{sub-Riemannian structure} on $M$ is a finite family of smooth vector fields $\F=\{X_1,\cdots, X_m\}\subset\mathfrak{X}(M)$ satisfying the \emph{H\"ormander condition}:
\begin{equation*}
    \forall p\in M,\  \mathrm{Lie}_p(\F)\coloneqq\{X(p)\mid X\in\mathrm{Lie}(\F)\}=T_pM,
\end{equation*}
where $\mathrm{Lie}(\F)\subset\mathfrak{X}(M)$ is the Lie algebra generated by $\F$. We refer to the pair $(M,\F)$ as a \emph{sub-Riemannian manifold}. The \emph{set $\D\subset\mathfrak{X}(M)$ of horizontal vector fields} is the $\mathcal{C}^{\infty}(M)$-module generated by $\F$ and $\D^i\coloneqq \D^{i-1}+[\D,\D^{i-1}]$ is defined inductively ($i\ge2$). Given a point $p\in M$, the \emph{distribution $\D_p\coloneqq\{X(p),X\in \D\}$ at $p$} comes equipped with the \emph{sub-Riemannian scalar product $\dotprod{\cdot}{\cdot}_p$ at $p$} (see \cite[Exercise 3.9]{ABB}), whose induced norm $\lvert \cdot\rvert_p$ satisfies: 
\begin{equation*}
    \forall v\in \D_p, \ \lvert v\rvert_p=\lvert v^*\rvert_{\R^m},
\end{equation*}
where $v^*\in\R^m$ is the \emph{minimal control of $v$}, i.e. the unique element such that:
\begin{equation*}
    \lvert v^*\rvert_{\R^m}=\mathrm{min}\Big\{\lvert u\rvert_{\R^m}\Big|\ u=(u_1,...,u_m)\in\R^m,v=\sum_{i=1}^mu_iX_i(p)\Big\}. 
\end{equation*}
The \emph{flag at a point $p\in M$} is the following sequence:
\begin{equation*}
    \{0\}\subset \D_p\subset \cdots\subset \D^{r(p)}_p=T_pM,
\end{equation*}
where  $ \D^i_p\coloneqq\{X(p),X\in \D^i\}$ and $r(p)$ is the \emph{degree of non-holonomy at $p$}, i.e. the smallest integer $i$ such that $\D^i_p=T_pM$. We denote $n_i(p)\coloneqq\dim(\D^i_p)$ ($i\ge1$) the \emph{flag dimensions at $p$} and, for  $n_{i-1}(p)+1\le j\le n_i(p)$, $\omega_j(p) \coloneqq  i$ the \emph{$j$-th weight at $p$} (where $n_0(p)\coloneqq 0$ by convention).

\begin{rem}
    Given a point $p\in M^n$ and any basis $\{v_1,\cdots, v_n\}$ of $T_pM$ adapted to the flag at $p$, $\omega_j(p)$ is the smallest integer $i$ such that $v_j\in \D^i_p$, or, in other words, the smallest number of Lie brackets of horizontal vector fields needed to create $v_j$.
\end{rem}

A curve $\gamma\colon I\to M$ is \emph{admissible} if there exists a \emph{control} $u\in L^{\infty}_{\mathrm{loc}}(I,\R^m)$ such that, for a.e. $t\in I$, we have:
\begin{equation*}
    \dot{\gamma}(t)=\sum_{i=1}^mu_i(t)X_i(\gamma(t)).
\end{equation*}
There always exists a unique \emph{minimal control} $u^*\in L_{\mathrm{loc}}^{\infty}(I,\R^m)$ such that $\lvert \dot{\gamma}(t)\rvert_{\gamma(t)}=\lvert u^*(t)\rvert_{\R^m}$ holds for a.e. $t\in I$. When the interval $I$ is finite, we denote:
\begin{equation*}
    \Le(\gamma)=\int_I\lvert\dot{\gamma}(t)\rvert_{\gamma(t)}\di t 
\end{equation*}
the \emph{length of $\gamma$}.

\begin{rem}
    The length functional is invariant under Lipschitz re-parametrization, and any admissible curve may be re-parametrized to have unit speed.
\end{rem}

Given two points $p,q\in M$, we define the \emph{sub-Riemannian distance $d_{\F}$ induced by $\F$}:
\begin{equation*}
    d_{\F}(p,q)\coloneqq\inf\{\Le(\gamma)|\ \gamma\colon[0,1]\to M \text{ admissible},\gamma(0)=p, \gamma(1)=q\}.
\end{equation*}
Thanks to the Chow--Rashevskii theorem (see \cite[Theorem 3.31]{ABB}), $d_{\F}$ metrizes the topology of $M$.
    
\begin{rem}\label{rem: most general definition of sub-Riemannian}
    A sub-Riemannian structure on a smooth manifold $M$ is sometimes defined as a pair $(U,f)$, where $U$ is a Euclidean vector bundle over $M$ and $f\colon U\to TM$ is a vector bundle homomorphism, such that the set of horizontal vector fields $\mathcal{D}\coloneqq\{f(\sigma),\sigma\text{ section of }U\}\subset\mathfrak{X}(M)$ satisfies the H\"ormander condition. However, thanks to \cite[Corollary 3.27]{ABB}, given such a pair $(U,f)$, there always exists a finite family $\F\subset\mathfrak{X}(M)$ satisfying the H\"ormander condition and such that $(M,\F)$ and $(U,f)$ are equivalent as sub-Riemannian structures (see \cite[Definition 3.18]{ABB}). Therefore, our definition aligns with the most general one presented in the literature.
\end{rem}

\subsection{Nilpotent approximation}

For the rest of this section, we fix a sub-Riemannian manifold $(M,\F)$ and a point $p\in M$, where $\F=\{X_1,\cdots, X_m\}$ is our sub-Riemannian structure. Thanks to \cite[Theorem 4.15]{Bellaiche_96}, there exists a system of \emph{privileged coordinates $\phi$ at $p$}, i.e. a system of coordinates $\phi = (x_1,\cdots,x_n)\colon U\to V\subset\R^n$ centered at $p$ such that the basis $\{\partial_{x_1}(p),\cdots,\partial_{x_n}(p)\}$ is adapted to the flag at $p$ and the order of $x_i$ at $p$  is exactly $\omega_i=\omega_i(p)$ (i.e. $x_i(q)=O_{q\to p}(  d_{\F}(p,q)^{\omega_i})$ and $x_i(q)\neq O_{q\to p}(  d_{\F}(p,q)^{\omega_i+1})$). Such privileged coordinates induce a \emph{$1$-parameter family of dilation $\{\delta_{\lambda}\}_{\lambda>0}$}:
\begin{equation*}
    \delta_{\lambda}(x)=(\lambda^{\omega_1}x_1,\cdots,\lambda^{\omega_n}x_n),
\end{equation*}
and a \emph{pseudo-norm} $\lVert x \rVert\coloneqq \lvert x_1\rvert^{1/{\omega_1}}+\cdots+\lvert x_n\rvert^{1/{\omega_n}}$, for $x=(x_1,\cdots,x_n)\in\R^n$.

\begin{rem}
    Thanks to our system of privileged coordinates, we may identify points $q\in U$ with their image $\phi(q)$, vector fields $X\in \mathfrak{X}(U)$ with their push-forward $\phi_{*}X$, and ${\dist_{\F}}_{\lvert U}$ with $\phi_{*}{\dist_{\F}}_{\lvert U}$. In particular, $p$ is identified with $0\in\R^n$.
\end{rem}

Thanks to \cite[Theorem 5.19]{Bellaiche_96}, the following decomposition holds for every $1\le i \le m$:
\begin{equation}\label{eq:nilp approx}
    X_i=\hat{X}_i+R_i,
\end{equation}
where $\delta_{\lambda}^*\hat{X}_i=\lambda^{-1}\hat{X}_i$ ($\lambda>0$) and $R_i(0)=0$. The \emph{nilpotent approximation of $\F$ at $p$} is the family $\hat{\F}\coloneqq\{\hat{X}_1,\cdots,\hat{X}_m\}\subset\mathfrak{X}(\R^n)$ which satisfies the H\"ormander condition (see \cite[Proposition 5.17]{Bellaiche_96}) and induce a sub-Riemannian distance $\hat{d}\coloneqq  d_{\hat{\F}}$ on $\R^n$.

\begin{rem}\label{rem: nilp distrib at 0 = distrib at p}
    Note that, identifying $p$ with $0$, we have $X_i(0)=\hat{X}_i(0)$ ($1\le i\le m$). In particular, $(M,\F)$ and $(\R^n,\hat{\F})$ have the same distribution at $0$ with the same scalar product.
\end{rem}

The following theorem will play a crucial role in the subsequent parts (see \cite[Theorem 10.65]{ABB} for a proof).

\begin{thm}\label{thm: uniform convergence of distances}
    The rescaled distances $\dist_{\lambda}\coloneqq\lambda\dist_{\F}(\delta_{\lambda^{-1}}\cdot,\delta_{\lambda^{-1}}\cdot)$ converge locally uniformly to $\hat{d}$ on $\R^n\times\R^n$ as $\lambda\to\infty$.
\end{thm}

\begin{rem}\label{rem: subR_tan_cone}
 As a result of Theorem \ref{thm: uniform convergence of distances}, $(M,\lambda\dist_{\F},p)$ converges in the pointed Gromov--Hausdorff topology to $(\R^n,\hat{\dist},0)$ as $\lambda\to\infty$ and the functions $\delta_{\lambda^{-1}}$ may act as our $\epsilon$-Gromov-Hausdorff approximations. Thus $(\mathbb{R}^n,\hat{d},0)$ is the (unique) tangent cone of $(M,d_{\F})$ at $p$.
\end{rem}

\subsection{Geodesics in sub-Riemannian manifolds and their blow-up}\label{sec:prelim_geodesic_in_sub_Riem_mflds}

There are two types of geodesics (i.e.\,\,constant-speed length-minimizing curves) in a sub-Riemannian manifold $(M,\F)$, namely \emph{normal} and \emph{abnormal geodesics}. A geodesic $\gamma$ is normal if it takes the form $\gamma=\pi(\lambda)$, where $\pi\colon T^*M\to M$ is the cotangent bundle projection, and $\lambda$ is a \emph{normal Pontryagin extremal}, i.e. an integral curve of the \emph{sub-Riemannian Hamiltonian vector field $\vec{H}\in\mathfrak{X}(T^*M)$} (see \cite[Definition 4.21]{ABB}). Normal geodesics are always smooth \cite[Theorem 4.25]{ABB}. A geodesic is called abnormal if it is not normal; contrary to normal geodesics, it may not be smooth.

Our first observation is that the distribution $\D_p$ at a point $p\in M$ is spanned by the velocities of normal geodesics through $p$.

\begin{prop}\label{prop: normal geodesics}
    For every $p\in M$ and $v\in \D_p$, there exists a normal geodesic $\gamma\colon(-\epsilon,\epsilon)\to M$ such that $\gamma(0)=p$ and $\dot{\gamma}(0)=v$.
\end{prop}

\begin{proof}
    We fix $\lambda_0\in T^*_pM$ such that $v=\sum_{i=1}^m\braket{\lambda_0|X_i(p)}X_i(p)$ and denote $\lambda(t)=\exp(t\vec{H})(\lambda_0)$. Thanks to \cite[Theorem 4.25]{ABB}, $\lambda$ is a normal Pontryagin extremal. Therefore, due to \cite[Theorem 4.65]{ABB}, there exists $\epsilon>0$ such that $\gamma(t)\coloneqq\pi(\lambda(t))$ ($t\in(-\epsilon,\epsilon)$) is a normal geodesic on $M$. As a result of \cite[equation (4.39)]{ABB}, we have $\dot{\gamma}(0)=\sum_{i=1}^m\braket{\lambda_0|X_i(p)}X_i(p)=v$, which concludes the proof.
\end{proof}

Below, we fix a point $p\in M$, and a system of privileged coordinates $\phi$ at $p$. Given a curve $\gamma\colon I\to M$ such that $0\in I$, $\gamma(0)=p$, and $\mathrm{Im}(\gamma)\subset \mathrm{Dom}(\phi)$, we introduce the \emph{rescaled curve}:
\begin{equation}\label{eq:rescaled curve}
    \gamma^{\lambda}\colon t\in\lambda I\mapsto \delta_{\lambda}(\gamma(\lambda^{-1}t))\in \R^n,
\end{equation}
where $\lambda>0$ and $\gamma$ is identified with its image by $\phi$. We call \emph{a blow-up of $\gamma$} any curve arising as the limit of $\gamma^{\lambda_k}$ as $k$ goes to $\infty$ in the topology of uniform convergence on compact subsets of $I_{\infty}=\cup_{\lambda>0}\lambda I$, for some sequence $\lambda_k\to\infty$ as $k$ goes to $\infty$. If $\gamma$ is a geodesic, then, as a result of Theorem \ref{thm: uniform convergence of distances}, any blow-up is a geodesic in $(\R^n,\hat{d})$. More precisely, when $\gamma$ is a normal geodesic, the blow-up is unique and can be described explicitly in terms of $\dot{\gamma}(0)$. For completeness, we include a proof in Appendix \ref{sec:appendix}.

\begin{prop}\label{prop: normal geodesic blow-up}
 If $\gamma\colon(-\epsilon,\epsilon)\to M$ is a unit-speed normal geodesic such that $\gamma(0)=p$ and $\dot{\gamma}(0)=v$, then $\gamma^{\lambda}$ converges locally uniformly on $\R$ to a line $\hat{\gamma}$ through $0$ in $(\R^n,\hat{\dist})$ as $\lambda$ goes to $\infty$. Moreover, for every $t\in\R$, we have $\hat{\gamma}(t)=e^{t\hat{v}}(0)$, where $\hat{v}=\sum_{i=1}^mv^*_i\hat{X}_i\in\mathfrak{X}(\R^n)$ and $v^*=(v_1^*,\cdots,v_m^*)\in\R^m$ is the minimal control of $v$.
\end{prop}

A simple consequence of the previous proposition is the following corollary (see Appendix \ref{sec:appendix} for a proof), which will play an important part in the proof of Theorem \ref{mainthm:not_CD}.
\begin{cor}\label{lem: normal_line}
    If we denote $(\R^n,\hat{\mathcal{F}})$ the nilpotent approximation of $M$ at $p$ and fix a vector $\hat{v}_0\in \hat{\mathcal{D}}_0$, then there exists a line $\hat{\gamma}\colon\R\to(\R^n,\hat{\dist})$ such that $\hat{\gamma}(0)=0$, $\frac{\dist}{\dist t}\hat{\gamma}(0)=\hat v_0$, and, for every $\lambda>0$, $\hat{\gamma}^{\lambda} = \hat{\gamma}$, where the rescaled curve $\hat{\gamma}^{\lambda}$ is introduced in \eqref{eq:rescaled curve}. In particular, $\hat{\gamma}$ is invariant under blow-up.
\end{cor}

If $\gamma$ is an abnormal geodesic, then it is not clear whether there is a unique blow-up. Monti, Pigati, and Vittone proved the following result \cite{Monti18} in the case of a constant-rank distribution. Their proof extends to the rank-varying case after some modifications. We provide a proof of the rank-varying case in Appendix \ref{sec:appendix}.

\begin{thm}\label{thm:geodesic blow-up}
    If $\gamma\colon[0,T]\to M$ is a unit-speed geodesics such that $\gamma(0)=p$, then there exists a sequence $\lambda_k\to\infty$ such that $\gamma^{\lambda_k}$ converges locally uniformly on $\R_{\ge0}$ to a ray $\hat{\gamma}$ emanating from $0$ in $(\R^n,\hat{\dist})$ as $k$ goes to $\infty$. Moreover, $\hat{\gamma}$ admits a constant control; in particular, it is smooth.
\end{thm}

\subsection{CD and RCD conditions}As a result of Gromov's precompactness theorem \cite[Theorem 5.3]{Gromov_98}, sequences of manifolds with dimension bounded above and Ricci curvature bounded below admit converging subsequences in the pointed Gromov--Hausdorff topology. Limits of such sequences, namely \emph{Ricci limit space}, were studied extensively since the seminal papers by Cheeger and Colding \cite{CC96,CCI,CCII,CCIII}. As Ricci limit spaces may be singular, the need for a generalized definition of Ricci curvature lower bounds became apparent. Inspired notably by \cite{Cordero-Erausquin_01}, Lott--Villani \cite{Lott_Villani_2009} and Sturm \cite{Sturm_I_06,Sturm_II_06} introduced the $\CDe(K,N)$ condition to characterize potentially non-smooth metric measure spaces with $\dim\le N$ and $\Ric\ge K$. 

A metric measure space (m.m.s.\,\,for short) consists of a triple $(X,d,\meas)$, where $(X,\dist)$ is a complete and separable metric space and $\meas$ is a nonnegative Radon measure on its Borel $\sigma$-algebra. We denote $\mathrm{Geo}(X,d)$ the set of constant-speed length-minimizing geodesics in $(X,d)$ parametrized on $[0,1]$. If $t\in[0,1]$, we denote $e_t\colon \gamma\in\mathrm{Geo}(X,d)\mapsto \gamma(t)\in X$ the time-$t$ evaluation map. Given $K\in\R$, $N\in(1,\infty)$, and $(t,\theta)\in[0,1]\times\R_{\geq0}$, we define:
    \begin{equation*}\label{eq: dist}
            \mathfrak{s}_{\kappa}(\theta)\coloneqq
	       \begin{cases*}
                \frac{\sin(\sqrt{\kappa}\theta)}{\sqrt{\kappa}}, & if $\kappa>0$\\
		  \theta, & if $\kappa=0$\\
		      \frac{\sinh(\sqrt{-\kappa}\theta)}{\sqrt{-       \kappa}}, & if $\kappa<0$
	\end{cases*} \text{ and } \sigma_{K,N}^{(t)}(\theta)\coloneqq
        \begin{cases*}
            \infty, & if $K\theta^2\geq N\pi^2$\\
            \frac{\mathfrak{s}_{K/N}(t\theta)}{\mathfrak{s}_{K/N}(\theta)}, & if $K\theta^2<N\pi^2$ and $K\theta^2\neq0$\\
		t, & if $K\theta^2=0$
		\end{cases*}.
    \end{equation*}
If $N>1$, we denote $\tau_{K,N}^{(t)}(\theta)\coloneqq t^{1/N}\{\sigma_{K,N-1}^{(t)}(\theta)\}^{1-1/N}$. 
We refer the readers to \cite{Villani_09} for an introduction to optimal transport, including Wasserstein geodesics and entropy functionals.

\begin{defn}
    A m.m.s.\,$(X,d,\meas)$ satisfies the \emph{$\CDe(K,N)$ condition} ($K\in\R$ and $N\in (1,\infty)$) if, given any pair of Borel probability measures $\mu_i=\rho_i\meas\ll \meas$ ($i=0,1$) with finite second moment, there exists a Borel probability measure $\eta$ on $\mathrm{Geo}(X,d)$ such that $\mu_t\coloneqq {e_t}_{\#}\eta$ ($0\le t\le1$) is an $\mathcal{L}^2$ Wassertein geodesic from $\mu_0$ to $\mu_1$ and, for every $N'\ge N$, we have the following property:
    \begin{equation*}
        \forall t\in[0,1],\ \mathcal{S}_{N'}(\mu_t\mid\meas)\le -\int_{X\times X} \Big[\tau_{K,N'}^{(1-t)}(d(x,y))\rho_0(x)^{-\frac{1}{N'}}+\tau_{K,N'}^{(t)}(d(x,y))\rho_1(x)^{-\frac{1}{N'}}\Big]\di \pi(x,y),
    \end{equation*}
    where $\mathcal{S}_{N'}(\cdot\mid\meas)$ denotes the R\'{e}nyi entropy with parameter $N'$ associated with $\meas$ and $\pi\coloneqq (e_0,e_1)_{\#}\eta$.
\end{defn}

\begin{rem}
    {If $(X,d,\meas)$ satisfies the $\CDe(K,N)$ condition, then $(\mathrm{Spt}(\meas),d)$ is a geodesic metric space (see for example \cite[Remark 4.18 (ii)]{Sturm_I_06}). In what follows, all measures have full support, and all metric spaces considered are geodesic.}
\end{rem}

Although Ricci limit spaces satisfy the $\CDe(K,N)$ condition (for adequate $K\in\R$ and $N\in(1,\infty)$), so do Finsler manifolds under an appropriate lower curvature bound (see \cite{Ohta_09}). However, Finsler manifolds arise as Ricci limit spaces only when they are Riemannian. In \cite{Ambrosio-Gigli-Savare_14}, Ambrosio, Gigli, and Savar\'e introduced the notion of $\RCD$ spaces, ruling out non-Riemannian Finsler examples.

\begin{defn}
    A m.m.s. $(X,d,\meas)$ satisfies the $\RCD(K,N)$ condition ($K\in\R$ and $N\in(1,\infty)$) if it satisfies the $\CDe(K,N)$ condition and is \emph{infinitesimally Hilbertian}, i.e. the Sobolev space $H^{1,2}(X,d,\meas)$ is a Hilbert space.
\end{defn}

Remarkably, Gigli proved in \cite{Gigli13,Gigli14} that $\RCD(0,N)$ spaces satisfy a splitting theorem, generalizing the Cheeger--Gromoll splitting theorem for Riemannian manifolds \cite{CG_split} and the Cheeger--Colding splitting theorem \cite{CC96} for Ricci limit spaces.  

\begin{thm}\label{thm: RCD splitting}
   Let $(X,d,\meas)$ be an $\RCD(0, N)$ space, where $N\in (1,\infty)$, such that $\meas$ has full support. If $(X,d)$ contains a line, then there exists a metric measure space $(X', d',\meas')$ such that $(X,d,\meas)$ is isomorphic to $(X', d',\meas')\otimes (\R,d_E,\mathcal{L}^1)$, where:
   \begin{itemize}
       \item $(X',d',\meas')$ is an $\RCD(0,N-1)$ space when $N\ge 2$,
       \item $(X',d',\meas')$ is a point when $N<2$,
   \end{itemize}
    and $\R$ is equipped with Euclidean distance $d_E$ and Lebesgue measure $\mathcal{L}^1$.
\end{thm}

\begin{rem}
    In the theorem above, an $\RCD(0,1)$ space should be understood as a line, ray, circle, segment, or point equipped with a constant multiple of its Hausdorff measure.
\end{rem}

We close this section with the following result stating the universal infinitesimal Hilbertianity of sub-Riemannian manifolds by Le Donne, Lu\v{c}i\'c, and Pasqualetto \cite[Theorem 1.2]{Le_Donne_23}.

\begin{thm}\label{thm: CD is RCD} 
Let $(M,\mathcal{F})$ be a sub-Riemannian manifold whose sub-Riemannian distance $d_{\mathcal{F}}$ is complete. If $\mathfrak{m}$ is a nonnegative Radon measure on $(M,d_{\mathcal{F}})$, then the m.m.s $(M,d_{\mathcal{F}},\mathfrak{m})$ is infinitesimally Hilbertian.
\end{thm}

\begin{rem}
    Theorem \ref{thm: CD is RCD} implies that, if $(M,\F)$ is sub-Riemannian manifold, and $\meas$ is a nonnegative Radon measure on $(M,d_{\mathcal{F}})$ such that $(M,d_{\F},\meas)$ is a $\CDe(K,N)$ space, then $(M,d_{\mathcal{F}},\mathfrak{m})$ is also an $\RCD(K,N)$ space.
\end{rem}

\section{Sub-Riemannian manifolds are universally non-CD}\label{sec:notRCD}

We prove our main result, Theorem \ref{mainthm:not_CD}, in this section. The argument proceeds by contradiction, following the strategy outlined in the introduction. We begin with a lemma on distance estimate for smooth admissible curves in a sub-Riemannian manifold. In Section \ref{subsec: cone split}, we prove that the blow-up of normal geodesics at a point $p$ gives rise to $n_1$ many independent lines in the tangent cone, where $n_1=\dim \mathcal{D}_p$. By applying Gigli's splitting theorem, we conclude that the tangent cone at $p$ must be isometric to a product $\mathbb{R}^{n_1}\times Z$ (Proposition \ref{prop: splitting}). Then, in Section \ref{subsec: proof of A}, we prove that this metric splitting is incompatible with the sub-Riemannian structure of the tangent cone. The desired contradiction arises from a blow-up of a geodesic in the $Z$-factor and the distance estimate in Section \ref{subsec: distance below}.

\subsection{Bounding from below by a Riemannian distance}\label{subsec: distance below}

In the case of a Riemannian manifold, one has an explicit Taylor expansion of the distance between smooth curves sharing the same base point. The situation is more subtle in the case of sub-Riemannian manifolds. Nevertheless, we can obtain an analogue result in the form of a first-order lower bound for the distance between smooth admissible curves.

The first step is to bound a sub-Riemannian distance from below by a Riemannian one.

\begin{lem}\label{lem: local Riemannian extension}
If $(M^n,\F)$ is a sub-Riemannian manifold, where $\F=\{X_1,\cdots,X_m\}$, and $p\in M$, then there exists a neighborhood $U$ of $p$ and a Riemannian metric $g$ on $U$ such that the restriction of $g_p$ to $\D_p$ coincides with the sub-Riemannian dot product $\dotprod{\cdot}{\cdot}_p$. Moreover, we have $\dist_g\leq\dist_{\F}$ on any ball $B^{\F}_{\epsilon}(p)$ such that $ B^{\F}_{2\epsilon}(p)\subset U$.
\end{lem}

\begin{proof}
    First, we fix $n_1=n_1(p)$ vector fields $Y_1,\cdots,Y_{n_1}\in\F$ such that $\{Y_1(p),\cdots,Y_{n_1}(p)\}$ forms a basis of $\D_p$. In particular, $\{Y_1,\cdots,Y_{n_1}\}$ are linearly independent in a neighborhood $U$ of $p$. Shrinking $U$ if necessary, there exist vector fields $X_{m+1},\cdots,X_{m+n-n_1}$ such that the family $\mathcal{G}\coloneqq\{Y_1,\cdots,Y_{n_1},X_{m+1},\cdots,X_{m+n-n_1}\}$ is a local frame on $U$. In particular, $(U,\mathcal{G})$ is Riemannian (see \cite[Exercise 3.24]{ABB}); we denote $g$ its underlying Riemannian metric. Now, assume that $x,y\in B^{\F}_{\epsilon}(p)$ and that $ B^{\F}_{2\epsilon}(p)\subset U$. In that case, there exists a length minimizing geodesic $\gamma$ in $(M,\dist_{\F})$ from $x$ to $y$ with values in $U$. Therefore, we have $\dist_g(x,y)\leq\mathcal{L}_{g}(\gamma)\leq \mathcal{L}_{\F}(\gamma)=\dist_{\F}(x,y)$, where the second equality holds since $\mathcal{G}$ contains $\F$. Finally, assume that $v\in \D_p$ and observe that:
    \begin{equation*}
        \lvert v\rvert_{g_p}^2=\min\Bigg\{\sum_{i=1}^{m+n-n_1}v_i^2\Bigg|\ v=\sum_{i=1}^{m+n-n_1}v_iX_i(p)\Bigg\}=\min\Bigg\{\sum_{i=1}^{m}v_i^2\Bigg|\  v=\sum_{i=1}^{m}v_iX_i(p)\Bigg\}=\lvert v\rvert_p^2,
    \end{equation*}
    where the first and last equalities hold by definition, and the second equality holds since $v\in \D_p$ and $\{X_{m+1}(p),\cdots,X_{m+n-1}(p)\}$ spans a complement of $\D_p$ in $T_pM$. Applying the polarization identity concludes the proof.
\end{proof}

\begin{rem}
     Note that since $\dist_{\F}$ metrizes the topology of $M$, there always exists $\epsilon>0$ such that $B^{\F}_{2\epsilon}(p)\subset U$ if $U$ is open.
\end{rem}

\begin{rem}
   Lemma \ref{lem: local Riemannian extension} implies that, fixing a system of privileged coordinates at $p$, there exists $C>0$ and $\delta>0$ such that, for every $0<\epsilon<\delta$, we have $B_{\epsilon}^{\mathcal{F}}(\epsilon)\subset C[-\epsilon,\epsilon]^n$. Such implication resembles a weak version of the ball-box Theorem; see \cite[Corollary 2.1]{jean_control_2014}. Note, however, that we do need a Riemannian metric $g$ such that $d_g\le d_{\mathcal{F}}$ in order to prove Corollary \ref{cor: distance between smooth admissible curves} below (which plays a key role in the proof of Theorem \ref{mainthm:not_CD}).
\end{rem}

We now obtain a lower bound on the sub-Riemannian distance between smooth admissible curves sharing the same base point.

\begin{cor}\label{cor: distance between smooth admissible curves}
    Let $(M,\F)$ be a sub-Riemannian manifold. If $\alpha,\beta\colon[0,\epsilon]\to M$ are smooth admissible curve such that $\alpha(0)=\beta(0)=p$, then $$\dist_{\F}(\alpha(t),\beta(t))\ge t\lvert\dot{\alpha}(0)-\dot{\beta}(0)\rvert_p+o_{t\to0}(t)$$ 
    holds for all $t\in[0,\epsilon]$.
\end{cor}

\begin{proof}
    Let us fix a Riemannian metric $g$ as in Lemma \ref{lem: local Riemannian extension} and observe that, for $t$ small enough, one has $\dist_{g}(\alpha(t),\beta(t))\leq\dist_{\F}(\alpha(t),\beta(t))$. For the Riemannian distance $\dist_g$, we have $$\dist_g(\alpha(t),\beta(t))=t\lvert\dot{\alpha}(0)-\dot{\beta(0)}\rvert_{g_p}+o_{t\to0}(t).$$ 
    Note that $g_p$ coincides with $\dotprod{\cdot}{\cdot}_p$ on $\D_p$. In addition, since $\alpha$ and $\beta$ are smooth and admissible, we have $\dot{\alpha}(0),\dot{\beta}(0)\in \D_p$. Therefore, $\lvert\dot{\alpha}(0)-\dot{\beta(0)}\rvert_{g_p}=\lvert\dot{\alpha}(0)-\dot{\beta(0)}\rvert_{p}$ and we obtain the desired inequality.
\end{proof}

\subsection{Splitting of the tangent cone by blow-up of normal geodesics}\label{subsec: cone split}

For the remainder of Section 3, let us fix a sub-Riemannian manifold $(M,\F)$, where $\F=\{X_1,\cdots,X_m\}$. We also assume that $(M,\dist_{\F})$ is equipped with a full-support nonnegative Radon measure $\meas$ such that $(M,d_{\F},\meas)$ is a $\CDe(K,N)$ space ($K\in\R$ and $N\in(1,\infty$)). We fix a point $p\in M$, a system of privileged coordinates $\phi$ at $p$, and denote $n_1=n_1(p)\coloneqq\dim(\D_p)$. Let us recall that $\hat{\F}=\{\hat{X}_1,\cdots,\hat{X}_m\}$ is the nilpotent approximation of $\F$ at $p$ and $\hat{d}=d_{\hat{\F}}$ is the induced sub-Riemannian distance on $\R^n$.

The goal of this section is to prove the following proposition, which splits an $\R^{n_1}$-factor in the tangent cone $(\R^n,\hat{d})$ of $(M,d_{\F})$ at $p$ by blowing up normal geodesics through $p$.

\begin{prop}\label{prop: splitting}
 There exists an isometry $\phi\colon(\R^{n_1}\times Z,(0,z))\to (\R^n,\hat{\dist},0)$ such that, for every $t\in\R$ and $1\le i \le n_1$, we have:
 \begin{equation*}
     \phi(te_i,z) = e^{t\hat{v}_i}(0),
 \end{equation*}
where $Z$ is a geodesic space, $\{v_1,\cdots,v_{n_1}\}$ is an orthonormal basis of $(\D_p,\dotprod{\cdot}{\cdot}_p)$, and the vector fields $\hat{v}_i\in\mathfrak{X}(\R^n)$ are introduced in Proposition \ref{prop: normal geodesic blow-up}.
\end{prop}

First of all, we equip $(\R^n,\hat{\dist})$ with a measure to turn it into an $\RCD(0,N)$ space.

\begin{lem}\label{lem: RCD(0,N)}
    There exists a full-support nonnegative Radon measure $\hat{\meas}$ on $(\R^n,\hat{\dist})$ such that $(\R^n,\hat{\dist},\hat{\meas})$ is an $\RCD(0,N)$ space.
\end{lem}

\begin{proof}
    First of all, thanks to Theorem \ref{thm: CD is RCD}, $(M,d_{\F},\meas)$ is an $\RCD(K,N)$ space. We fix a sequence $r_i\to0$ as $i\to\infty$, denote $\meas_i\coloneqq\meas(B_{r_i}(p))^{-1}\meas$, and observe that, since $\meas$ has full-support, then $(M,r_i^{-1}d_{\F},\meas_i,p)$ is a pointed, full-support, normalized $\RCD(r_i^2K,N)$ space. As a result of Gromov's precompactness theorem \cite[Theorem 5.3]{Gromov_98} and the stability of $\RCD(K,N)$ spaces under pmGH convergence \cite[Theorem 7.2]{AGS_convergence} (after \cite{Ambrosio-Gigli-Savare_14,Villani_09,Sturm_I_06,Sturm_II_06}), we may assume, passing to a subsequence is necessary, that $(M,r_i^{-1}d_{\F},\meas_i,p)\to (Y,d_Y,\meas_Y,y)$ in the pmGH topology as $i\to\infty$, where $(Y,d_Y,\meas_Y,y)$ is a pointed, full-support, normalized $\RCD(0,N)$ spaces. By Theorem \ref{thm: uniform convergence of distances} and Remark \ref{rem: subR_tan_cone}, $(Y,d_Y,y)$ is isometric to $(\R^n,\hat{d},0)$, which concludes the proof.
\end{proof}

\begin{notn}
    Given a unit-speed normal geodesic $\gamma$ through $p$, we denote $\hat{\gamma}$ its blow-up, whose existence and uniqueness are stated by Proposition \ref{prop: normal geodesic blow-up}.
\end{notn}

The following two lemmas are the main ingredients in the proof of Proposition \ref{prop: splitting}.

\begin{lem}\label{lem: angle estimate}
    If $\alpha,\beta\colon(-\epsilon,\epsilon)\to M$ are unit-speed normal geodesics such that $\alpha(0)=\beta(0)=p$, then for all $t\in\R$, we have $$\hat{d}(\hat{\alpha}(\pm t),\hat{\beta}(t))\ge \lvert t\rvert \sqrt{2\mp 2\cos\theta},$$ 
    where $\theta$ denotes the angle between $\dot{\alpha}(0)$ and $\dot{\beta}(0)$.
\end{lem}

\begin{proof}
    Observe that, for $t\in(-\epsilon,\epsilon)$, Corollary \ref{cor: distance between smooth admissible curves} implies that 
    $$\dist_{\F}(\alpha(t),\beta(t))\ge \lvert t\rvert \sqrt{2-2\cos\theta}+ \epsilon(t),$$
    where $\lim_{t\to0}\epsilon(t)/t=0$. Moreover, we have $\hat{\dist}(\hat{\alpha}(t),\hat{\beta}(t))=\lim_{\lambda\to\infty}\dist_{\lambda}(\alpha^{\lambda}(t),\beta^{\lambda}(t))$ (thanks to Proposition \ref{prop: normal geodesic blow-up}). Given $t\in\R$ fixed and $\lambda>0$, we have: 
    \begin{equation*}
        \dist_{\lambda}(\alpha^{\lambda}(t),\beta^{\lambda}(t))=\lambda\dist_{\F}(\alpha(t/\lambda),\beta(t/\lambda))\ge\lvert t\rvert \sqrt{2-2\cos\theta}+\lambda\epsilon(t/\lambda).
    \end{equation*}
    Since $\lim_{\lambda\to\infty}\lambda\epsilon(t/\lambda)=0$, we obtain $\hat{d}(\hat{\alpha}(t),\hat{\beta}(t))\ge \lvert t\rvert \sqrt{2-2\cos\theta}$. Using the same arguments, we also show $\hat{d}(\hat{\alpha}(-t),\hat{\beta}(t))\ge \lvert t\rvert \sqrt{2+2\cos\theta}$.
\end{proof}

\begin{lem}\label{lem: splitting angle}
 If $\hat{\alpha}$ and $\hat{\beta}$ are lines through $0$ in $(\R^n,\hat{\dist})$ such that:
\begin{equation*}
   \forall t\in\R,\  \hat{d}(\hat{\alpha}(\pm t),\hat{\beta}(t))\ge \lvert t\rvert \sqrt{2\mp 2\cos\theta},
\end{equation*}
and if $\phi\colon(\R,0)\times(Z,z)\to(\R^n,\hat{\dist},0)$ is an isometry such that, for all $t\in\R$, we have $\phi(t,z)=\hat{\alpha}(t)$, then, denoting $\phi(\beta_{\R}(t),\beta_{Z}(t))=\hat{\beta}(t)$, we have $\beta_{\R}(t)=t\cos\theta$, for all $t\in\R$.
\end{lem}

\begin{proof}
   Since $\hat{\beta}$ is a unit-speed geodesic emanating at $(0,z)$, we have $\dist_{Z}^2(z,\beta_{Z}(t))+\lvert \beta_{\R}(t)\rvert^2 = t^2$. In particular, since $$\hat{d}^2(\hat{\alpha}(t),\hat{\beta}(t))\ge t^2 (2-2\cos\theta),$$
   we have:
   \begin{align*}
       t^2(2-2\cos\theta) \le& (\beta_\R(t)-t)^2 + \dist_{Z}^2(z,\beta_{Z}(t))\\
       =&\dist_{Z}^2(z,\beta_{Z}(t))+\lvert \beta_{\R}(t)\rvert^2+t^2-2t\beta_{\R}(t)\\ 
       =&2t^2-2t\beta_{\R}(t),
   \end{align*}
   which implies $- 2\beta_\R(t)\cdot t \ge -2t^2 \cos\theta$. Similarly, from the other distance estimate $$\hat{d}^2(\hat{\alpha}(-t),\hat{\beta}(t))\ge t^2 (2+2\cos\theta),$$
   we derive $2\beta_\R(t)\cdot t \ge 2t^2 \cos\theta$. We obtain $t^2\cos\theta=\beta_\R(t)\cdot t$, and the result follows.
\end{proof}

We may now prove Proposition \ref{prop: splitting}.

\begin{proof}[Proof of Proposition \ref{prop: splitting}]
We fix an orthonormal basis $\{v_1,\cdots,v_{n_1}\}$ of $(\D_p,\dotprod{\cdot}{\cdot}_p)$. Thanks to Proposition \ref{prop: normal geodesics}, for each $i=1,...,n_1$, there exist a unit-speed normal geodesic $\gamma_i\colon(-\epsilon,\epsilon)\to M$ such that $\gamma_i(0)=p$ and $\dot{\gamma}_i(0)=v_i$. Thanks to Lemma \ref{lem: RCD(0,N)}, there exists a full-support nonnegative Radon measure $\hat{\meas}$ on $(\R^n,\hat{d})$ such that $(\R^n,\hat{d},\hat{\meas})$ is an $\RCD(0,N)$ space.

We will proceed by induction on $k$ and prove that, for $1\le k\le n_1$, there exists an isomorphism of pointed m.m.s. $\phi_k\colon(\R^k,0,d_E,\mathcal{L})\otimes (Z_k,d_{Z_k},\meas_k,z_k)\to(\R^n,\hat{\dist},\hat{\meas},0)$ such that, for all $t\in\R$ and $1\le i \le k$, we have $\phi_k(te_i,z_k)=\hat{\gamma}_i(t)$, where $(Z_k,d_{Z_k},\meas_k)$ is a full-support $\RCD(0,N-k)$ space.

By Proposition \ref{prop: normal geodesic blow-up}, the curves $\hat{\gamma}_i$ are lines in $(\R^n,\hat{\dist})$; thus, thanks to the splitting theorem \cite{Gigli14}, our induction hypothesis holds for $k=1$.

Now, assume that our induction hypothesis holds for some $1\le k<n_1$ and let us construct $\phi_{k+1}$. The initial tangent vectors of $\hat{\gamma}_{k+1}$ and $\hat{\gamma}_{i}$ has angle $\pi/2$ between them. Thanks to Lemma \ref{lem: angle estimate}, for all $t\in\R$, we have $\hat{\dist}(\hat{\gamma}_{k+1}(t),\hat{\gamma}_i(\pm t))\ge\sqrt{2}\lvert t\rvert$. Therefore, denoting $\hat{\gamma}_{k+1}=\phi(\beta_{\R}^1,\cdots,\beta_{\R}^k,\beta_Z)$, we have $\beta_{\R}^1=\cdots=\beta_{\R}^k=0$ as a result of Lemma \ref{lem: splitting angle}. In particular, thanks to the splitting theorem \cite{Gigli14}, there exists an isomorphism of pointed m.m.s. 
$$\psi\colon (\R,0,d_E,\mathcal{L})\otimes (Z_{k+1},d_{Z_{k+1}},\meas_{k+1},z_{k+1})\to (Z_k,d_{Z_k},\meas_k,z_k)$$
such that $\psi(t,z_{k+1})=\beta_Z(t)$ and $(Z_{k+1},d_{Z_{k+1}},\meas_{k+1})$ is a full-support $\RCD(0,N-k-1)$ space. Therefore, the isomorphism:
\begin{equation*}
    \phi_{k+1}((t_1,\cdots,t_{k+1}),z')\coloneqq \phi_k((t_1,\cdots,t_k),\psi(t_{k+1},z')),
\end{equation*}
satisfies the inductive step, which concludes the proof.
\end{proof}

\subsection{Proof of universal non-CD}\label{subsec: proof of A}

We use Proposition \ref{prop: splitting} and Theorem \ref{thm:geodesic blow-up} to complete the proof of Theorem \ref{mainthm:not_CD}.

\begin{proof}[Proof of Theorem \ref{mainthm:not_CD}]

Looking for a contradiction, let us assume $n_1\coloneqq n_1(p)<n\coloneqq\dim M$ for some $p\in M$. Let us fix an isometry $\phi\colon(\R^{n_1},0)\times(Z,z)\to(\R^n,\hat{\dist},0)$ as in Proposition \ref{prop: splitting}, where $(\R^n,\hat{\dist})$ is the nilpotent approximation of $M$ at $p$.

    If $Z$ consisted of a single point, then $\phi$ would induce a homeomorphism between $\R^{n_1}$ and $\R^n$, which would imply $n=n_1$, a contradiction. Since $Z$ is a geodesic space and consists of at least two points, we may fix a non-constant unit-speed geodesic $\gamma\colon[0,\epsilon]\to Z$ such that $\gamma(0)=z$. We will identify $\gamma$ with $\phi(0,\gamma)$.

    Noting that the nilpotent approximation of $(\R^n,\hat{d})$ at $0$ is itself, we apply Theorem \ref{thm:geodesic blow-up} to $\gamma$ in $(\R^n,\hat{d})$. Then there exists a sequence $\lambda_k\to\infty$ such that $\gamma^{\lambda_k}$ converges locally uniformly on $\R_{\ge0}$ to a ray $\gamma_{\infty}$ with constant control in $(\R^n,\hat{\dist})$ as $k\to \infty$. In particular, $\gamma_{\infty}\colon\R_{\ge0}\to(\R^n,\hat{\dist})$ is a smooth unit-speed ray satisfying:
    \begin{equation*}
     \begin{cases*}
         \gamma_{\infty}(0)=0\\
         \dot{\gamma}_{\infty}(t)=\sum_{i=1}^mu_i\hat{X}_{i}(\gamma_{\infty}(t))
     \end{cases*},
    \end{equation*}
    where $u=(u_1,\cdots,u_m)\in\R^m$ is a constant control.

    Let us denote $v\coloneqq\sum_{i=1}^mu_i\hat{X}_i(0)\in \hat{\D}_0$. Thanks to Corollary \ref{lem: normal_line}, there is a normal geodesic ${\gamma}_v$ in $(\mathbb{R}^n,\hat{d})$ with ${\gamma}_v(0)=0$ and $\dot{\gamma}_v(0)=v$ such that ${\gamma}_v$ is a line and is invariant under blow-ups. 
    

    We claim that ${\gamma}_v$ is contained in $\mathbb{R}^{n_1}\times \{z\}$. We fix an orthonormal basis $\{e_1,\cdots,e_{n_1}\}$ of $\hat{\D}_0$. For each $i=1,...,n_1$, we denote $\theta_i$ the angle between $v$ and $e_i$, and apply Corollary \ref{lem: normal_line} to obtain a normal line $\beta_i$ in $(\R^n,\hat{d})$ that fulfills the conclusion of Corollary \ref{lem: normal_line}. Since both $\gamma_v$ and $\beta_i$ are normal geodesics and are invariant under blow-ups of $(\mathbb{R}^n,\hat{d})$, by Lemma \ref{lem: angle estimate}, we have the distance estimate
    $$\hat{d}(\gamma_v(\pm t),\beta_i(t))\ge |t|\sqrt{2\mp 2\cos\theta}.$$
    Thanks to Lemma \ref{lem: splitting angle}, we can write $\gamma_v$ as
    $${\gamma}_v(t)=\phi(t\cos(\theta_1),\dots, t\cos(\theta_{n_1}),\gamma_Z(t)).$$ 
    Because ${\gamma}_v$ is a line in $(\R^n,\hat{\dist})$, we have:
    \begin{equation}
        t^2 =\sum_{i=1}^{n_1} t^2 \cos^2(\theta_i) + d_Z^2(\gamma_Z(t),z) \ge t^2 \sum_{i=1}^{n_1} \cos^2(\theta_i) = t^2,
    \end{equation}
    where the last equality holds because $v\in \hat{\D}_0$ is a unit vector and $\{e_1,\cdots,e_{n_1}\}$ is an orthonormal basis of $\hat{\D}_0$. This shows that $\gamma_Z(t)\equiv z$ is constant, and thus verifies the claim that ${\gamma}_v$ is contained in $\mathbb{R}^{n_1}\times \{z\}$.

    Now, we have $\gamma\subset \{0\}\times Z$ and ${\gamma}_v\subset\R^{n_1}\times\{z\}$, two unit speed geodesics emanating at $(0,z)$. Due to the metric product $\R^{n_1}\times Z$, we obtain
    \begin{equation}\label{eq: distance between gamma and gamma_v}
        \forall t\in[0,\epsilon],\ \hat{\dist}({\gamma}_v(-t),\gamma(t))=\sqrt{2}t.
    \end{equation}
    Moreover, $\gamma_v$ is invariant under blow-ups by its construction from Corollary \ref{lem: normal_line}. Therefore, for every $t\in\R_{\ge0}$, the following holds:
    \begin{equation*}
        \hat{\dist}(\gamma_v(-t),\gamma_{\infty}(t))=\lim_{k\to\infty}\hat{\dist}(\gamma_v^{\lambda_k}(-t),\gamma^{\lambda_k}(t))=\lim_{k\to\infty}\lambda_k\hat{\dist}(\gamma_v(-t/\lambda_k),\gamma(t/\lambda_k))=\sqrt{2}t,
    \end{equation*}
    using \eqref{eq: distance between gamma and gamma_v}. On the other hand, by construction we have:
    \begin{equation*}
        \dot{\gamma}_{v}(0)=v=\sum_{i=1}^mu_i\hat{X}_i(0)=\dot{\gamma}_{\infty}(0).
    \end{equation*} 
    Thanks to Corollary \ref{cor: distance between smooth admissible curves}, the following holds for all $t\ge0$:
    \begin{equation}\label{eq: contradiction}
        \sqrt{2}t=\hat{\dist}(\gamma_v(-t),\gamma_{\infty}(t))\ge2t+o_{t\to0^+}(t),
    \end{equation}
    using $\dot{\gamma}_{v}(0)=v=\dot{\gamma}_{\infty}(0)$ and $\lvert v\rvert_0=1$. Letting $t\to0^+$ in \eqref{eq: contradiction} gives $\sqrt{2}\ge2$, which is the contradiction we were looking for, hence concluding the proof. 
\end{proof}

Next, we prove Corollary \ref{cor: subFinsler_not_RCD}. We refer the reader to \cite{Magnabosco_Rossi_Review_2025} for an introduction to smooth sub-Finsler manifolds and to \cite{A_LD_NG_23} for a proof of an analogue of Theorem \ref{thm: uniform convergence of distances} in the general case of Lipschitz sub-Finsler manifolds.

\begin{proof}[Proof of Corollary \ref{cor: subFinsler_not_RCD}]
Let $d_{sF}$ be the distance on $M^n$ induced by a sub-Finsler structure $(\xi,\lvert\cdot\rvert)$ (see \cite[Section 2]{Magnabosco_Rossi_Review_2025} for the notations) and assume that $(M,d_{sF},\meas)$ is an $\RCD(K,N)$ space for some nonnegative full-support Radon measure $\meas$, $K\in\R$, and $N\in(1,\infty)$.

First, observe that the set $\mathcal{R}\subset M$ of points $p\in M$ which are equiregular (i.e., the flag dimensions $n_i$ are locally constant near $p$) and weakly regular (i.e., there exists a metric measure tangent cone at $p$ which is isomorphic to the Euclidean space) is dense in $M$. Indeed, the set of equiregular points is open and dense in $M$ (see the third bullet point following \cite[Example 2.6]{jean_control_2014}). Moreover, the set of weakly regular points has full measure (see \cite[Theorem 1.1]{gigli_euclidean_2015}); hence, it is dense. In particular, $\mathcal{R}$ is a dense subset of $M$ as the intersection of a dense open subset with a dense subset.

Let us fix $p\in \mathcal{R}$. Thanks to \cite[Theorem 1.5]{A_LD_NG_23}, the tangent cone at $p$ is a (unique) sub-Finsler Carnot group $(\R^n,\hat{d}_p)$, where $\hat{d}_p$ is induced by the nilpotent approximation $(\hat{\xi}^p,\lvert\cdot\rvert_p)$ on $\R^n$ (see \cite[Definition 3.3]{Magnabosco_Rossi_Review_2025}). Since $p$ is weakly regular, its (unique) metric tangent cone $(\R^n,\hat{d}_p)$ is isometric to a Euclidean space and thus has Hausdorff dimension $n$; as a consequence, $(\R^n,\hat{d}_p)$ is Finsler (see \cite[Corollary 4.3.6]{LeD_book}). Since $(\R^n,\hat{d}_p)$ is isometric to a Euclidean space, and since isometries between Finsler manifolds are Finsler isometries (see \cite[Theorem 10]{aradi_isometries_2014}), the nilpotent approximation at $p$ is sub-Riemannian. Therefore, $\lvert\cdot\rvert_p$ satisfies the parallelogram identity whenever $p\in \mathcal{R}$. 

Since $\mathcal{R}\subset M$ is dense and $\lvert\cdot\rvert$ is continuous, the norm $\lvert\cdot\rvert_p$ satisfies the parallelogram identity for every $p\in M$. Hence, $(M,d_{sF})$ is sub-Riemannian and we conclude using Theorem \ref{mainthm:not_CD}.
\end{proof}

\section{New examples of RCD structures on $\mathbb{R}^n$}\label{sec:exmp}

We prove Theorem \ref{mainthm:exmp} in this section.

We first define a distance $d$ on $\mathbb{R}^n$ as follows. Let 
$$C=\{(0,y)\in \mathbb{R}^{k+1} \times \mathbb{R}\ |\ y\in\mathbb{R}\}$$
be a curve in $\mathbb{R}^{k+2}$, where $k\ge 2$. On the complement of $C$, that is, 
$$\Omega:=\mathbb{R}^{k+2}-C=\{(x,y)\in \mathbb{R}^{k+1}\times \mathbb{R}\ | \ x\not= 0 \},$$
let us consider an incomplete doubly-warped Riemannian metric defined by $(0,\infty)\times_{cr} S^k \times_{r^{-\alpha}} \mathbb{R}$:
$$dr^2 + (cr)^2 ds_k^2 + r^{-2\alpha} dy^2,$$
where $\alpha>0$, $c\in(0,1)$, and $ds_k^2$ denotes the standard round metric on $S^k$. Here, we identify $\mathbb{R}^{k+1}$ with a cone over $S^k$. We write $n=k+2$ and denote $(\mathbb{R}^n,\dist)$ the metric completion of $(\Omega,\dist_g)$. 

We call the above defined $(\mathbb{R}^n,d)$ a \textit{cone-Grushin space} since it is the doubly-warped product of a cone metric (with warping function $cr$) and a Grushin metric (with warping function $r^{-\alpha}$).

Below we construct the cone-Grushin space, where $c\in(0,1)$ is sufficiently small, as the (unique) asymptotic cone of a complete manifold with positive Ricci curvature. In particular, this shows that the cone-Grushin space is an RCD$(0,N)$ space for some $N<\infty$.

On $\mathbb{R}^{m+1} \times S^k \times S^1$, we consider a triply-warped product
$$M=[0,\infty)\times_f S^m \times_g S^k \times_h S^1, \quad dr^2+f(r)^2ds_m^2 +g(r)^2ds_k^2 + h(r)^2 ds_1^2.$$
We use warping functions
$$f(r)=\dfrac{r}{(1+r^2)^{1/4}},\quad g(r)=\dfrac{\pi}{2}\cdot\dfrac{cr}{\arctan r},\quad h(r)=(1+r^2)^{-\alpha/2},$$
where $c\in(0,1)$ is a small constant to be determined later. We remark that the warping functions $f$ and $h$ were also used in \cite{Pan-Wei_2022,Wei_1988}. We note that $f,g,h$ satisfy the following properties:
$$f(0)=0, \quad f^{(\text{even})}(0)=0,\quad 0<f'<1,\quad f''\le 0;$$
$$ g(0)>0,\quad g^{(\text{odd})}(0)=0, \quad \lim\limits_{r\to\infty}r^{-1} g(r)=c,\quad 0\le g'<c; $$
$$ h(0)>0,\quad h^{(\text{odd})}(0)=0,\quad h'<0.$$
They define a smooth Riemannian metric on $\mathbb{R}^{m+1} \times S^k \times S^1$.

\begin{lem}
For each $k\ge 2$ and $\alpha>0$, we can choose suitable $m\ge 2$ and $c\in (0,1)$ such that $M$ has positive Ricci curvature. 
\end{lem}

\begin{proof} Let $X,Y$, and $Z$ be a unit vector tangent to the components $S^m$, $S^k$, and $S^1$, respectively, where $k\ge 2$ is fixed and $m\ge 2$ will be determined later. Then, by direction calculation, we have the Ricci curvature (see \cite[Section 4.2.4]{Pet_book})
\begin{align*}
    \Ric(\partial_r,\partial_r)&=-m\dfrac{f''}{f}-k\dfrac{g''}{g}-\dfrac{h''}{h},\\
    \Ric(X,X)&= -\dfrac{f''}{f}+(m-1)\dfrac{1-(f')^2}{f^2}-k\dfrac{f'g'}{fg}-\dfrac{f'h'}{fh},\\
    \Ric(Y,Y)&= -\dfrac{g''}{g}-m\dfrac{f'g'}{fg}+(k-1)\dfrac{1-(g')^2}{g^2}-\dfrac{g'h'}{gh},\\
    \Ric(Z,Z)&= -\dfrac{h''}{h}-m\dfrac{f'h'}{fh}-k\dfrac{g'h'}{gh}.
\end{align*}
We note that the value of $c\in (0,1)$ is only involved in the computation of $\Ric(Y,Y)$. By direct computation, we have
$$\dfrac{f''}{f}=-\dfrac{x^2+6}{4(x^2+1)^2},\quad \dfrac{h''}{h}=\alpha \dfrac{(\alpha+1)x^2-1}{(x^2+1)^2},\quad \dfrac{f'h'}{fh}=-\alpha \dfrac{x^2+2}{2(x^2+1)^2};$$
$$-\dfrac{f''}{f}-\dfrac{g''}{g}>0,\quad \dfrac{1-(f')^2}{f^2}-\dfrac{f'g'}{fg}>0.$$
To ensure that $\Ric(\partial_r,\partial_r)$, $\Ric(X,X)$, and $\Ric(Z,Z)$ are positive, we can choose $m$ large such that $m>\max\{ k+4\alpha(\alpha+1),k+1,2(\alpha+1) \}$. Next, we pick $c\in (0,1)$ to obtain $\Ric(Y,Y)>0$. Since
$$-\dfrac{f'g'}{fg}<-\dfrac{g''}{g}<0,\quad -\dfrac{g'h'}{gh}>0,$$
it suffices to find $c\in (0,1)$ such that
\begin{equation}\label{eq: Ric_Y}
-(m+1)\dfrac{f'g'}{fg}+(k-1)\dfrac{1-(g')^2}{g^2}>0.
\end{equation}
We calculate
$$\dfrac{f'}{f}=\dfrac{r^2+2}{2r(r^2+1)}<\dfrac{1}{r},\quad \dfrac{g'}{g} = \dfrac{(r^2+1)\arctan r-r}{r(r^2+1)\arctan r},\quad \dfrac{1}{g}=\dfrac{2}{\pi}\cdot \dfrac{\arctan r}{cr}.$$
On $[0,1]$, using $\arctan r\le r$, we estimate
$$\dfrac{g'}{g}\le \dfrac{r^2}{(r^2+1)\arctan r};$$
on $[1,\infty]$, we have
$$\dfrac{g'}{g}\le \dfrac{1}{r}.$$
From (\ref{eq: Ric_Y}), if we choose $c\in (0,1)$ small such that
$$  
(k-1)\cdot\dfrac{4}{c^2\pi^2}\cdot \dfrac{\arctan^2 r}{r^2} > \begin{cases} \dfrac{(m+1)r}{(r^2+1)\arctan r} + \dfrac{(k-1)r^4}{(r^2+1)^2\arctan^2 r} &\text{ when } r\in [0,1]\\
  \dfrac{m+1}{r^2}+\dfrac{k-1}{r^2}
&\text{ when } r\in [1,\infty]
\end{cases},$$
then $\Ric(Y,Y)>0$. This completes the proof.
\end{proof}

\begin{prop}
  Let $\widetilde{M}$ be the Riemannian universal cover of the above-defined $M$. Then $\widetilde{M}$ has a unique asymptotic cone as the cone-Grushin space $(\mathbb{R}^{k+2},d)$.
\end{prop}

\begin{proof}
It follows from the construction of $M$ that $\widetilde{M}$ has a Riemannian metric as a triply-warped product
$$[0,\infty)\times_{f} S^m \times_g S^k \times_h \mathbb{R},\quad dr^2+f(r)^2ds_m^2 +g(r)^2ds_k^2 + h(r)^2 dv^2.$$
$\widetilde{M}$ is diffeomorphic to $\R^{m+1}\times S^k \times \R$.

For any $\lambda>1$, we apply a change of variables $t=\lambda^{-1} r$ and $w=\lambda^{-1-\alpha} v$, then 
\begin{align*}
\lambda^{-2}\tilde{g}&=\lambda^{-2}[dr^2+f^2(r)ds_m^2 +g^2(r)ds_k^2 + h(r)^2 dv^2]\\
&=dt^2 + \dfrac{t^2}{(1+\lambda^2 t^2)^{1/2}}ds_m^2 + \dfrac{\pi^2}{4}\cdot \dfrac{c^2t^2}{\arctan^2 (\lambda t)}ds_k^2 + \dfrac{\lambda^{2\alpha}}{(1+\lambda^2t^2)^{\alpha}} dw^2\\
&=dt^2 + f_\lambda(t)^2 ds_m^2 + g_\lambda(t)^2 ds_k^2 + h_\lambda(t)^2 dw^2.
\end{align*}
As $\lambda\to\infty$, we have
$$f_\lambda \to 0,\quad g_\lambda\to ct,\quad h_\lambda \to t^{-\alpha}$$
converge uniformly on every compact subset of $(0,\infty)$. Hence, the Gromov-Hausdorff limit space of $\lambda^{-1}\widetilde{M}$ as $\lambda \to \infty$ is the cone-Grushin space with metric
$$dt^2+(ct)^2 ds_k^2 + t^{-2\alpha} dw^2.$$
\end{proof}

\begin{proof}[Proof of Theorem \ref{mainthm:exmp}]
(1) We have shown that the cone-Grushin space $(\mathbb{R}^n,d)$ is the (unique) asymptotic cone of some complete Riemannian manifold of positive Ricci curvature and dimension $N=m+k+2$. Endowed with a limit renormalized measure $\meas$, the metric measure space $(\mathbb{R}^n,d,\meas)$ is an RCD$(0,N)$ space.

(2) It is clear that $\Delta_q=T_q \R^n$ for any point $q\in \R^n-C$. At every singular point $(0,y)\in C$, we have horizontal directions
$$\Delta_{(0,y)} = \mathrm{span}\{ \partial_{x_1},...,\partial_{x_{k+1}} \},$$ 
where $(x_1,...,x_{k+1})$ is the standard coordinate of $x\in \mathbb{R}^{k+1}$. To see that $\Delta$ coincides with a distribution that is generated by a finite family of smooth vector fields with the H\"ormander condition, one may use
$$X_1=\partial_{x_1},\ ...,\  X_{k+1}=\partial_{x_{k+1}},\  X_{k+2}=r^2 \partial_y, $$
where $r^2=\sum_{i=1}^{k+1} x_i^2$.

(3) We construct metric dilations on $(\mathbb{R}^n,d)$. For every $\lambda>0$, we define
$$\delta_\lambda:\R^{k+2} \to \R^{k+2},\quad (x,y)\to (\lambda x, \lambda^{1+\alpha} y),$$
where $x\in \R^{k+1}$ and $y\in \R$. One can verify that the pullback metric satisfies
$$\delta_\lambda^*(dr^2 + (cr)^2 ds_k^2 + r^{-2\alpha} dy^2)=\lambda^2(dr^2 + (cr)^2 ds_k^2 + r^{-2\alpha} dy^2).$$
Hence $\delta_\lambda$ is a metric dilation with factor $\lambda$, that is,
$$d(\delta_\lambda(x,y),\delta_\lambda(x',y'))=\lambda\cdot d((x,y),(x',y'))$$
for all $(x,y),(x',y')\in \mathbb{R}^{k+2}$.

(4) We note that the metric space $(\mathbb{R}^n,d)$ has symmetries as translations and reflections in the $y$-direction. We set $c=d((0,0),(0,1))>0$. For any point $(0,y)\in C$ with $y\not=0$, we can use the metric dilation $\delta_\lambda$, where $\lambda=|y|^{\frac{1}{1+\alpha}}$, to compute
$$d((0,0),(0,y))=d((0,0),(0,|y|))=d(\delta_\lambda(0,0),\delta_\lambda(0,\pm 1))=c\cdot |y|^{\frac{1}{1+\alpha}}.$$
Consequently, $(C,d)$ has Hausdorff dimension $1+\alpha$.
\end{proof}

\begin{rem}
   We showed that the cone-Grushin space is a Ricci limit space. Alternatively, one may directly work with $(\mathbb{R}^{n},d)$ and equip it with a measure
   $$\meas = r^p \dvol_g,$$ 
   where $p>1$ large and $\dvol_g$ is the Riemannian volume from the cone-Grushin metric. One can  compute that the $N$-Barky-{\'E}mery curvature on $\Omega$ is positive for small $c\in (0,1)$ and large $p,N$. However, it is unclear to the authors whether $\Omega$ is always geodesically convex in $(\mathbb{R}^n,d)$. When $\Omega$ is geodesically convex, one can apply the argument in \cite[Section 3.5]{Rizzi-Stefano_2023} to conclude that $(\mathbb{R}^n,d,\meas)$ is RCD$(0,N)$; also see \cite[Theorem 4.1]{Borza_Tashiro_24}.
\end{rem}

\appendix

\section{Blow-up of geodesics in sub-Riemannian manifolds}\label{sec:appendix}
We prove Proposition \ref{prop: normal geodesic blow-up} and Theorem \ref{thm:geodesic blow-up} in this appendix. In Appendix \ref{subsec: smooth minimal control}, we show that the minimal control of a normal geodesic is smooth, which will be used in the subsequent subsection. After studying blow-ups of normal geodesics in Appendix \ref{subsec: normal blowup}, we focus on Theorem \ref{thm:geodesic blow-up} (proved by Monti--Pigati--Vittone in \cite{Monti18} for constant-rank distributions). To extend their result to the rank-varying case, we use a desingularization that lifts the nilpotent approximation to a Carnot group, which is described in \cite[Section 5.4]{Bellaiche_96}; for readers' convenience, we conclude the details in Appendix \ref{subsec: Carnot lift}. We will then be able to conclude the proof of Theorem \ref{thm:geodesic blow-up} in Appendix \ref{subsec: abnormal blowup} thanks to the following theorem by Monti--Pigati--Vittone (see \cite[Proof of Theorem 1.1]{Monti18}).

\begin{thm}\label{thm:Carnot blow-up}\cite{Monti18}
Let $G$ be a Carnot group and denote $\{\delta_{\lambda}\}_{\lambda>0}$ its family of dilations. If $\gamma\colon I\to G$ is length-minimizing and satisfies $0\in I$ and $\gamma(0)=e$ (where $e$ denotes the identity element of $G$), then there exists a sequence $\lambda_k\to\infty$ such that $\gamma^{\lambda_k}$ (see \eqref{eq:rescaled curve}) converges locally uniformly to a length-minimizing curve $\hat{\gamma}$ with constant control.
\end{thm}

Below, we fix a sub-Riemannian structure $\F=\{X_1,\cdots,X_m\}$ on a smooth manifold $M$, a point $p\in M$, and a system of privileged coordinates $\phi$ at $p$. We recall that $\hat{\F}=\{\hat{X}_1,\cdots,\hat{X}_m\}$ denotes the nilpotent approximation of $\F$ at $p$ and $\hat{d}={d}_{\hat{\F}}$ its induced distance on $\R^n$.

\subsection{Smoothness of the minimal control for normal geodesics}\label{subsec: smooth minimal control}

For the sake of completeness, we provide a proof that the minimal control of a normal geodesic is smooth (see Section \ref{sec:prelin_sub-riem} for the definition of the minimal control).

\begin{lem}\label{lem:minimal_control_smooth}
    If $\gamma\colon(-\epsilon,\epsilon)\to M$ is a normal geodesic, then its minimal control is smooth.
\end{lem}

Before proving Lemma \ref{lem:minimal_control_smooth}, we recall a few simple linear algebra facts.

\begin{lem}\label{lem: linear algebra}
    Let $E$ be a finite dimensional vector space, $f=\{x_1,\cdots, x_m\}$ a family of vectors in $E$, denote $F=\mathrm{Span}_{\R}(f)$, introduce the maps:
\begin{equation*}
    \begin{cases*}
       A((u_1,\cdots,u_m)\in\R^m)=\sum_{i=1}^m u_i x_i \in E\\
       \phi(\lambda \in E^*) = \sum_{i=1}^m \lambda(x_i)x_i\in E\\
       A^{T}(\lambda\in E^*) =(\lambda(x_1),\cdots,\lambda(x_m))\in \R^m
    \end{cases*},
\end{equation*}
and denote $\lvert\cdot\rvert_{\F}$ the norm on $F$ defined by $\lvert v \rvert_{\F} \coloneqq \min\{\lvert u\rvert_{\R^m}, A(u) = v\}$ ($v\in F$). Then, $\phi(\lambda) = v$ implies $\lvert v\rvert_{\F} = \lvert A^T(\lambda)\rvert_{\R^m}$.
\end{lem}

\begin{rem}
    Observe that $A^T$ is the dual of $A$ (i.e. $\dotprod{A^T(\lambda)}{u}_{\R^m} = \lambda(A(u))$, for $u\in\R^m$); in particular, by duality, we have $\mathrm{Im}(A^T)=\mathrm{Ker}(A)^{\perp}$. The fact that $\lvert \cdot \rvert_{\F}$ is an inner product follows from the parallelogram identity.
\end{rem}

\begin{proof}
First of all, observe that $\phi = A\circ A^T$. Given $v\in F$, there exists $u\in\R^m$ such that $A(u)=v$. Decompose $u$ as $u = u_{\mathrm{row}} + u_{\mathrm{null}}$, where $u_{\mathrm{row}}\in \mathrm{Ker} (A)^\perp=\mathrm{Im}(A^T)$ and $u_{\mathrm{null}}\in \mathrm{Ker} (A)$. In particular, $A(u_{\mathrm{row}})=A(u)=v$. We have $u_{\mathrm{row}} = A^T(\lambda)$ for some $\lambda\in E^*$. Therefore, we have $v = A(u_{\mathrm{row}})=A(A^T(\lambda))=\phi(\lambda)$; which implies that $\mathrm{Im}(\phi) = F$.

Now, let us assume that $\phi(\lambda)=0$. Since $\phi = A\circ A^T$, we have $A^T(\lambda)\in \mathrm{Ker}(A)\cap\mathrm{Im}(A^T)$. However $\mathrm{Im}(A^T)=\mathrm{Ker}(A)^{\perp}$, so we necessarily have $A^T(\lambda)=0$, i.e. $\lambda\in F^{\perp}$; thus $\mathrm{Ker}(\phi) = F^{\perp}$.

Let us fix $v\in F$ and $u\in\R^m$ such that $A(u)=v$ and $\lvert u\rvert_{\R^m}=\lvert v\rvert_{\F}$. Observe that $A(u)=A(u_{\mathrm{row}})$ and $\lvert u_{\mathrm{row}}\rvert_{\R^m}\le \lvert u\rvert_{\R^m}$, where $u_{\mathrm{row}}$ is the orthogonal projection of $u$ onto $\mathrm{Ker} (A)^\perp=\mathrm{Im}(A^T)$. In particular, $\lvert v\rvert_{\F}\le \lvert u_{\mathrm{row}}\rvert_{\R^m}$, by definition of $\lvert \cdot\rvert_{\F}$. Therefore, since $\lvert u\rvert_{\R^m}=\lvert v\rvert_{\F}$, we necessarily have $u=u_{\mathrm{row}}$. Now observe that $A$ induces an isomorphism from $\mathrm{Ker}(A)^{\perp}=\mathrm{Im}(A)^T$ onto $F$. In particular, if we have $\phi(\lambda)=v$ (i.e.\, $A(A^T(\lambda))=A(u_{\mathrm{row}})$), then we have $A^T(\lambda) = u_{\mathrm{row}}$; hence $\lvert v\rvert_{\F} = \lvert u_{\mathrm{row}}\rvert_{\R^m} = \lvert A^T(\lambda)\rvert_{\R^m}$.
\end{proof}

\begin{proof}[Proof of Lemma \ref{lem:minimal_control_smooth}]
    As a result of \cite[Theorems 3.59 and 4.25]{ABB}, since $\gamma$ is a normal geodesic, it is smooth and satisfies:
    \[
    \dot{\gamma}(t) = \sum_{i=1}^m \dotprod{\lambda(t)}{X_i(\gamma(t))}X_i(\gamma(t)),
    \]
    where $\lambda(t) = \exp(t\Vec{H})(\lambda_0)\in T_{\gamma(t)}^*M$ ($t\in\R$, $\lambda_0\in T_{\gamma(0)}^*M$) is a normal Pontryagin extremal. In particular, thanks to Lemma \ref{lem: linear algebra}, $|\dot{\gamma}(t)|_{\gamma(t)} = |(\dotprod{\lambda(t)}{X_i(\gamma(t))})_{1\le i\le m}|_{\R^m}$. Therefore, by definition, $(\dotprod{\lambda}{X_i(\gamma})_{1\le i\le m}$ is the minimal control of $\gamma$. Finally, $(\dotprod{\lambda}{X_i(\gamma})_{1\le i\le m}$ is smooth since $\lambda$, $\gamma$, and $X_i$'s are smooth.
\end{proof}

\subsection{Blow-up of normal geodesics}\label{subsec: normal blowup}

\begin{proof}[Proof of Proposition \ref{prop: normal geodesic blow-up}]
    Observe that, making $\epsilon$ smaller if necessary, we may assume that $\gamma$ is length-minimizing and $\mathrm{Im}(\gamma)\subset\mathrm{Dom}(\phi)$, where $\phi$ is our system of privileged coordinates and $p$ is identified with $0$. Since $\gamma$ is a normal geodesic, we have $\dot{\gamma}(t)=\sum_{i=1}^mu_i(t)X_i(\gamma(t))$, where $u=(u_1,\cdots,u_m)$ is the minimal control of ${\gamma}$, which is smooth thanks to Lemma \ref{lem:minimal_control_smooth}. For every $\lambda>0$ and $t\in(-\lambda\epsilon,\lambda\epsilon)$, the following holds:
    \begin{equation*}
        \dot{\gamma}^{\lambda}(t)=\sum_{i=1}^mu_i^{\lambda}(t)X_i^{\lambda}(\gamma^{\lambda}(t)),\text{ where }X_i^{\lambda}\coloneqq\lambda^{-1}{\delta_{\lambda}}_{*}X_i\text{ and }u_i^{\lambda}\coloneqq u_i(\lambda^{-1}\cdot).
    \end{equation*}
    Observe that $\lim_{\lambda\to\infty}X_i^{\lambda}=\hat{X}_i$ (resp. $\lim_{\lambda\to\infty}u_i^{\lambda}=u_i(0)$) in the topology of smooth convergence on compact sets thanks to \cite[Lemma 10.58]{ABB} (resp. since the functions $u_i$ are smooth). We also have $\gamma^{\lambda}(0)=0$ for every $\lambda>0$. As a result of the continuity of solutions to ODEs w.r.t. to parameters (see \cite[Theorem 3.2]{Hartman}), $\gamma^{\lambda}$ converges locally uniformly on $\R$ to $\hat{\gamma}\colon t\in\R\to \exp(t\hat{v})(0)\in\R^n$. Finally, thanks to Theorem \ref{thm: uniform convergence of distances}, $\hat{\gamma}$ is a line in $(\R^n,\hat{\dist})$.
\end{proof}

Blow-ups of normal geodesics satisfy the following dilation identity.

\begin{lem}\label{lem: blow-up = diltation}
    Assume that $\hat{Y}$ satisfies $\delta_{\lambda}^*\hat{Y}=\lambda^{-1}\hat{Y}$ ($\lambda>0$), then we have the following property:
    \begin{equation*}
        \forall \lambda>0,\ \forall t\in\R,\ e^{t\hat{Y}}(0)=\delta_{\lambda}(e^{t\lambda^{-1}\hat{Y}}(0)).
    \end{equation*}
    In particular, given any $v\in \D_p$ and $t\in\R$, we have $e^{t\hat{v}}(0)=\delta_t(e^{\hat{v}}(0))$, where $\hat{v}\in\mathfrak{X}(\R^n)$ is introduced in Proposition \ref{prop: normal geodesic blow-up}.
\end{lem}

\begin{proof}
 Since  $\delta_{\lambda}(0)=0$  and $\delta_{\lambda}^*\hat{Y}=\lambda^{-1}\hat{Y}$, we have $e^{t\delta_{\lambda}^*\hat{Y}}(0)=\delta_{\lambda}^{-1}(e^{t\hat{Y}}(\delta_{\lambda}(0)))$  by definition of a pull-back vector field. We conclude by observing that if $v\in \D_p$, then $\hat{v}$ satisfies $\delta_{\lambda}^*\hat{v}=\lambda^{-1}\hat{v}$, where $\lambda>0$.
\end{proof}

Now we prove Corollary \ref{lem: normal_line}.

\begin{proof}[Proof of Corollary \ref{lem: normal_line}] Recall that, by Remark \ref{rem: nilp distrib at 0 = distrib at p}, $(\mathcal{D}_p,\langle\cdot,\cdot\rangle_p)$ is isometric to $(\hat{\mathcal{D}}_0,\langle\cdot,\cdot\rangle_0)$. As a result, $\hat v_0$ gives rise to a vector $v\in \mathcal{D}_p$ with identical minimal control $(v_1^*,\dots,v_m^*)$. By Proposition \ref{prop: normal geodesics}, there exists a normal geodesic $\gamma \colon (-\varepsilon,\varepsilon)\to M$ such that $\gamma(0)=p$ and $\dot\gamma(0)=v$. Thanks to Proposition \ref{prop: normal geodesic blow-up}, blowing up $\gamma$ gives rise to a line $\hat\gamma \colon \mathbb{R}\to(\mathbb{R}^n,\hat d)$ satisfying $\hat\gamma(t)=e^{t\hat v}(0)$, where $\hat v=\sum_{i=1}^m v_i^*\,\hat X_i \in \mathfrak X(\mathbb{R}^n)$. In particular, $\hat\gamma(0)=0$ and
\[
\frac{d}{dt}\hat\gamma(0)
=\hat v(0)
=\sum_{i=1}^m v_i^*\,\hat X_i(0)
=\hat v_0,
\]
where the last equality follows since $v$ and $\hat v_0$ have identical minimal control. Finally, recall that the vector fields $\hat X_i$ $(1\le i\le m)$ satisfy $\delta_\lambda^*\hat X_i=\lambda^{-1}\hat X_i$. Hence, thanks to Lemma \ref{lem: blow-up = diltation}, for every $t\in\mathbb{R}$ and $\lambda>0$, we have $\delta_\lambda\bigl(e^{t\hat v}(0)\bigr)=e^{\lambda t\hat v}(0)$. Consequently, for every $\lambda>0$ and $t\in\mathbb{R}$, $\hat\gamma_\lambda(t)
=\delta_\lambda(e^{\lambda^{-1}t\hat v}(0))
=e^{t\hat v}(0)
=\hat\gamma(t)$, which concludes the proof.
\end{proof}

\subsection{Carnot group lifting}\label{subsec: Carnot lift}

Thanks to \cite[Proposition 5.17]{Bellaiche_96}, $\mathfrak{g}_p\coloneqq\mathrm{Lie}_{\R}(\hat{X}_1,\cdots,\hat{X}_m)\le \mathfrak{X}(\R^n)$ is a finite-dimensional stratified Lie algebra and $G_p\coloneqq\{\exp(X),X\in \mathfrak{g}_p\}\le\mathrm{Diff}(\R^n)$ is its associated Carnot group. The sub-Riemannian structure of $G_p$ is generated by the family of left-invariant vector fields $\xi_1,\cdots,\xi_m\in\mathfrak{X}(G_p)$  defined in the following way:
\begin{equation*}
   \forall g \in G_p,\ \xi_i(g)\coloneqq\frac{d}{dt}_{\lvert t=0}g\circ \exp(-t\hat{X}_i)=d_eL_g(-\hat{X}_i).
\end{equation*}
Since every pair of points in $(\R^n,\hat{d})$ may be joined by an admissible curve with piecewise constant control (see the last remark of \cite[Section 2.5]{Bellaiche_96}), the map $\pi\colon g\in G_p\to g^{-1}(0)\in \R^n$ is surjective. We remark that one may alternatively use the natural right action of $G_p$ on $\mathbb{R}^n$.

\begin{lem}\label{lem: push forward vector in Carnot lift}
	The following properties hold:\\
    (1) $\pi_*\xi_i=\hat{X}_i$;\\
    (2) $\delta_{\lambda}\circ\pi=\pi\circ\hat\delta_{\lambda}$,\\
    where $\hat\delta_{\lambda}$ denotes the dilation on $G_p$ arising from the stratification $\mathfrak{g}_p=\oplus_{i=1}^{r}\mathfrak{g}_p^i$.
\end{lem}

\begin{proof}
    (1) Let us fix $g\in G_p$ and observe that:
    \begin{equation*}
        d_g\pi(\xi_i(g))=\frac{d}{dt}_{\lvert t=0}\pi(g\circ \exp(-t\hat{X}_i))=\frac{d}{dt}_{\lvert t=0}\exp(t\hat{X}_i)\circ g^{-1}(0)=\hat{X}_i(g^{-1}(0))=\hat{X}_i(\pi(g)).
    \end{equation*}
    (2) Now assume $\lambda>0$ and fix any element $g=\exp(X)\in \exp(\mathfrak{g}_p^1)$. We have:
    \begin{equation*}
        \pi(\hat\delta_{\lambda}g)=\pi(\exp(\lambda X))=\exp(-\lambda X)(0)=\delta_{\lambda}(\exp(-X)(0))=\delta_{\lambda}(\pi(g)),
    \end{equation*}
    where the second-to-last equality holds thanks to Lemma \ref{lem: blow-up = diltation}. To conclude, by \cite[Lemma 1.40]{FS_book}, $G_p$ is generated by $\exp(\mathfrak{g}_p^1)$; thus, any $g\in G_p$ can be expressed as $g=g_1...g_k$ with $g_i\in\exp(\mathfrak{g}_p^1)$. We prove (2) by induction on $k$ and suppose it holds for $k-1$. Denoting $h=g_1...g_{k-1}$, and observing that $\hat\delta_{\lambda}(g)=\hat\delta_\lambda(h)\hat\delta_\lambda(g_k)$, we have:
    \begin{equation*}
    \pi(\hat\delta_{\lambda}(g))=\hat{\delta}_\lambda(g_k^{-1}) \cdot \pi(\hat{\delta}_\lambda(h))=\hat{\delta}_\lambda(g_k^{-1}) \cdot \delta_\lambda(\pi(h))=\delta_\lambda(g_k^{-1}\cdot \pi(h))=\delta_{\lambda}(\pi(g)),
    \end{equation*}
where we used the induction hypothesis for the second equality and the property $\hat{\delta}_{\lambda}(g)\cdot\delta_{\lambda}(x)=\delta_{\lambda}(g\cdot x)$ ($x\in\R^n$) for the third equality.
\end{proof}

\begin{defn}
   Let $\gamma\colon I\to \R^n$ be an admissible curve in $(\R^n,\hat\F)$ with minimal control $u\in L^{\infty}(I,\R^m)$. A horizontal lift of $\gamma$ is an admissible curve $\overline{\gamma}\colon I\to G_p$ admitting $u$ as a control (w.r.t. $\{\xi_1,\cdots,\xi_m\}$) and such that $\pi(\overline{\gamma})=\gamma$.
\end{defn}

\begin{prop}\label{prop: horizontal lift}
    Let $\gamma\colon I\to \R^n$ be an admissible curve in $(\R^n,\hat\F)$. Then the following hold.\\
    (1) $\gamma$ admits a horizontal lift $\bar{\gamma}:I\to G_p$.\\
    (2) If $\gamma$ is length-minimizing, then any horizontal lift is also length-minimizing.\\
    (3) If there is an admissible curve $\overline{\gamma}$ in $G_p$ with constant control such that $\pi(\overline{\gamma})=\gamma$, then $\gamma$ admits a constant control.
\end{prop}

\begin{proof}
    (1) Let $\gamma:I\to \mathbb{R}^n$ be an admissible curve in $(\R^n,\hat\F)$ with minimal control $u\in L^{\infty}(I,\R^m)$ and let us fix $t_0\in I$. Since $\pi$ is surjective, there exists $g\in G_p$ such that $\pi(g)=\gamma(t_0)$. Thanks to \cite[Theorem 2.15]{ABB}, there exists a unique curve $\overline{\gamma}\colon I\to G_p$ such that $\overline{\gamma}(t_0)=g$ and, for a.e. $t\in I$, we have:
    \begin{equation*}
        \dot{\overline{\gamma}}(t)=\sum_{i=1}^mu_i(t)\xi_i(\overline{\gamma}(t)).
    \end{equation*}
    If we denote $\alpha\coloneqq\pi(\overline{\gamma})$, then $\alpha(t_0)=\gamma(t_0)$ and, for a.e. $t\in I$, we have:
    \begin{equation*}
        \dot{\alpha}(t)=d_{\overline{\gamma}(t)}\pi(\dot{\overline{\gamma}}(t))=\sum_{i=1}^mu_i(t)d_{\overline{\gamma}(t)}\pi(\xi_i(\overline{\gamma}(t)))=\sum_{i=1}^mu_i(t)\hat{X}_i(\pi(\overline{\gamma}(t)))=\sum_{i=1}^mu_i(t)\hat{X}_i(\alpha(t)),
    \end{equation*}
    where we used Lemma \ref{lem: push forward vector in Carnot lift} (1) for the third equality. Therefore, thanks to \cite[Theorem 2.15]{ABB}, we have $\gamma(t)=\alpha(t)=\pi(\overline{\gamma}(t))$ for every $t\in I$, i.e. $\overline{\gamma}$ is a horizontal lift of $\gamma$.
    
    (2,3) Now, assume that $\gamma$ is length-minimizing on the interval $[a,b]$ and that $\overline{\gamma}$ is a horizontal lift. We need to show that $\overline{\gamma}$ is also length-minimizing on $[a,b]$. Observe that, by definition of the sub-Riemannian distance on $G_p$, and since $u$ is the minimal control of $\gamma$, we have:
    \begin{equation}\label{eq: lift decreases distance}
        \dist_{G_p}(\overline{\gamma}(a),\overline{\gamma}(b))\le \mathcal{L}(\overline{\gamma})\le \int_a^b\lvert u\rvert_{\R^m}=\mathcal{L}(\gamma)=\hat\dist(\gamma(a),\gamma(b)).
    \end{equation}
    Now, let $\overline{\alpha}\colon[a,b]\to G_p$ be an admissible length-minimizing curve from $\overline{\gamma}(a)$ to $\overline{\gamma}(b)$ with minimal control $v\in L^{\infty}([a,b],\R^m)$. In particular, $\alpha\coloneqq\pi(\overline{\alpha})$ satisfies $\alpha(a)=\gamma(a)$ and $\alpha(b)=\gamma(b)$. Moreover, thanks to Lemma \ref{lem: push forward vector in Carnot lift} (1), $\alpha$ is admissible in $(\R^n,\hat\dist)$ with control $v$ (which also proves part (3) of the proposition). Therefore, we have:
    \begin{equation}\label{eq: eq: lift increases distance}
        \hat\dist(\gamma(a),\gamma(b))\le\mathcal{L}(\alpha)\le\int_a^b\lvert v\rvert_{\R^m}=\dist_{G_p}(\overline{\gamma}(a),\overline{\gamma}(b)).
    \end{equation}
    As a result of \eqref{eq: lift decreases distance} and \eqref{eq: eq: lift increases distance}, we have $\dist_{G_p}(\overline{\gamma}(a),\overline{\gamma}(b))=\mathcal{L}(\overline{\gamma})$, i.e. $\overline{\gamma}$ is length-minimizing on $[a,b]$.
\end{proof}

\subsection{Blow-up of abnormal geodesics}\label{subsec: abnormal blowup}

\begin{proof}[Proof of Theorem \ref{thm:geodesic blow-up}]
    First, thanks to the theorem of Arzel\`a--Ascoli  and Theorem \ref{thm: uniform convergence of distances}, there exists a sequence $\eta_k\to\infty$ such that $\gamma^{\eta_k}$ converges locally uniformly to a ray $\hat{\gamma}$ emanating from $0$ in $(\R^n,\hat{d})$. Thanks to Proposition \ref{prop: horizontal lift}, $\hat{\gamma}$ admits a horizontal lift $\overline{\gamma}\colon \R_{\ge0}\to G_p$ which is also length-minimizing and satisfies $\overline{\gamma}(0)=e\in G$. Using Theorem \ref{thm:Carnot blow-up}, there exists a sequence $\xi_k\to\infty$ such that $\overline{\gamma}^{\xi_k}$ converge locally uniformly to a length-minimizing curve $\overline{\gamma}_{\infty}\colon \R_{\ge0}\to G_p$ with constant control and $\overline{\gamma}_{\infty}(0)=e$. Lemma \ref{lem: push forward vector in Carnot lift}(2) together with $\pi(\overline{\gamma})=\hat{\gamma}$ implies:
    \begin{equation*}
        \pi(\overline\gamma^{\xi_k})=\pi(\hat\delta_{\xi_k}(\overline{\gamma}(\xi_k^{-1}\cdot))={\delta}_{\xi_k}(\pi(\overline{\gamma}(\xi_k^{-1}\cdot)))={\delta}_{\xi_k}(\hat{\gamma}(\xi_k^{-1}))=\hat{\gamma}^{\xi_k}.
    \end{equation*}
    Hence, $\hat{\gamma}^{\xi_k}$ converge locally uniformly to the ray $\hat{\gamma}_{\infty}\coloneqq\pi(\overline{\gamma}_{\infty})$, which satisfies $\hat{\gamma}_{\infty}(0)=0$. Since $\overline{\gamma}_{\infty}$ admits a constant control, the same holds for $\hat{\gamma}_{\infty}$ thanks to Proposition \ref{prop: horizontal lift}(3). Finally, proceeding with a diagonal argument as in \cite[Proposition 3.7]{Monti17}, there exists a sequence $\lambda_k\to\infty$ such that $\gamma^{\lambda_k}$ converges locally uniformly to $\hat{\gamma}_{\infty}$, which concludes the proof.
\end{proof}

\bibliographystyle{abbrv}
\bibliography{ref.bib}

@article{Rizzi-Stefano_2023,
	AUTHOR = {Rizzi, Luca and Stefani, Giorgio},
     TITLE = {Failure of curvature-dimension conditions on sub-{R}iemannian
              manifolds via tangent isometries},
   JOURNAL = {J. Funct. Anal.},
  FJOURNAL = {Journal of Functional Analysis},
    VOLUME = {285},
      YEAR = {2023},
    NUMBER = {9},
     PAGES = {Paper No. 110099, 31},
      ISSN = {0022-1236,1096-0783},
   MRCLASS = {53C17 (28A75 54E45)},
  MRNUMBER = {4623954},
       DOI = {10.1016/j.jfa.2023.110099},
       URL = {https://doi.org/10.1016/j.jfa.2023.110099},
}

@article {Pan-Wei_2022,
    AUTHOR = {Pan, Jiayin and Wei, Guofang},
     TITLE = {Examples of {R}icci limit spaces with non-integer {H}ausdorff
              dimension},
   JOURNAL = {Geom. Funct. Anal.},
  FJOURNAL = {Geometric and Functional Analysis},
    VOLUME = {32},
      YEAR = {2022},
    NUMBER = {3},
     PAGES = {676--685},
      ISSN = {1016-443X,1420-8970},
   MRCLASS = {53C23 (53C21)},
  MRNUMBER = {4431126},
MRREVIEWER = {Lina\ Chen},
       DOI = {10.1007/s00039-022-00598-4},
       URL = {https://doi.org/10.1007/s00039-022-00598-4},
}

@article {Pan_2023,
    AUTHOR = {Pan, Jiayin},
     TITLE = {The {G}rushin hemisphere as a {R}icci limit space with
              curvature {$\ge1$}},
   JOURNAL = {Proc. Amer. Math. Soc. Ser. B},
  FJOURNAL = {Proceedings of the American Mathematical Society. Series B},
    VOLUME = {10},
      YEAR = {2023},
     PAGES = {71--75},
      ISSN = {2330-1511},
   MRCLASS = {53C23 (53C20)},
  MRNUMBER = {4566173},
MRREVIEWER = {Raquel\ Perales},
       DOI = {10.1090/bproc/160},
       URL = {https://doi.org/10.1090/bproc/160},
}

@article {DHPW_2023,
    AUTHOR = {Dai, Xianzhe and Honda, Shouhei and Pan, Jiayin and Wei,
              Guofang},
     TITLE = {Singular {W}eyl's law with {R}icci curvature bounded below},
   JOURNAL = {Trans. Amer. Math. Soc. Ser. B},
  FJOURNAL = {Transactions of the American Mathematical Society. Series B},
    VOLUME = {10},
      YEAR = {2023},
     PAGES = {1212--1253},
      ISSN = {2330-0000},
   MRCLASS = {53C23 (53C17 53C21)},
  MRNUMBER = {4634191},
MRREVIEWER = {Yaoting\ Gui},
       DOI = {10.1090/btran/160},
       URL = {https://doi.org/10.1090/btran/160},
}

@article {Wei_1988,
    AUTHOR = {Wei, Guofang},
     TITLE = {Examples of complete manifolds of positive {R}icci curvature
              with nilpotent isometry groups},
   JOURNAL = {Bull. Amer. Math. Soc. (N.S.)},
  FJOURNAL = {American Mathematical Society. Bulletin. New Series},
    VOLUME = {19},
      YEAR = {1988},
    NUMBER = {1},
     PAGES = {311--313},
      ISSN = {0273-0979,1088-9485},
   MRCLASS = {53C20},
  MRNUMBER = {940494},
MRREVIEWER = {Karsten\ Grove},
       DOI = {10.1090/S0273-0979-1988-15653-4},
       URL = {https://doi.org/10.1090/S0273-0979-1988-15653-4},
}

@book{Hartman,
	series = {Classics in {Applied} {Mathematics}},
	title = {Ordinary differential equations},
	isbn = {0-89871-510-5},
	publisher = {Society for Industrial and Applied Mathematics},
	author = {Hartman, Philip},
	year = {2002},
}

@article {Monti17,
    AUTHOR = {Monti, Roberto and Pigati, Alessandro and Vittone, Davide},
     TITLE = {On tangent cones to length minimizers in {C}arnot-{C}arath\'eodory spaces},
   JOURNAL = {SIAM J. Control Optim.},
  FJOURNAL = {SIAM Journal on Control and Optimization},
    VOLUME = {56},
      YEAR = {2018},
    NUMBER = {5},
     PAGES = {3351--3369},
      ISSN = {0363-0129,1095-7138},
   MRCLASS = {53C17 (28A75 49K30)},
  MRNUMBER = {3857882},
MRREVIEWER = {Jeremy\ T.\ Tyson},
       DOI = {10.1137/17M114056X},
       URL = {https://doi.org/10.1137/17M114056X},
}

@Article{Gigli13,
Author = {Nicola Gigli},
Title = {The splitting theorem in non-smooth context},	
Journal = {arXiv:1302.5555},
Year = {2013},
}

@article{Gigli14,
author = {Nicola Gigli},
journal = {Analysis and {G}eometry in {M}etric {S}paces},
language = {eng},
number = {1},
pages = {169--213},
title = {An Overview of the Proof of the Splitting Theorem in Spaces with Non-Negative {R}icci Curvature},
url = {http://eudml.org/doc/266707},
volume = {2},
year = {2014},
}

@article {Monti18,
    AUTHOR = {Monti, Roberto and Pigati, Alessandro and Vittone, Davide},
     TITLE = {Existence of tangent lines to {C}arnot-{C}arath\'eodory
              geodesics},
   JOURNAL = {Calc. Var. Partial Differential Equations},
  FJOURNAL = {Calculus of Variations and Partial Differential Equations},
    VOLUME = {57},
      YEAR = {2018},
    NUMBER = {3},
     PAGES = {Paper No. 75, 18},
      ISSN = {0944-2669,1432-0835},
   MRCLASS = {53C17 (28A75 49K30)},
  MRNUMBER = {3795207},
MRREVIEWER = {Gareth\ Speight},
       DOI = {10.1007/s00526-018-1361-7},
       URL = {https://doi.org/10.1007/s00526-018-1361-7},
}

@book {FS_book,
    AUTHOR = {Folland, G. B. and Stein, Elias M.},
     TITLE = {Hardy spaces on homogeneous groups},
    SERIES = {Mathematical Notes},
    VOLUME = {28},
 PUBLISHER = {Princeton University Press, Princeton, NJ; University of Tokyo
              Press, Tokyo},
      YEAR = {1982},
     PAGES = {xii+285},
      ISBN = {0-691-08310-X},
   MRCLASS = {43A85 (22E45 42B30)},
  MRNUMBER = {657581},
MRREVIEWER = {Daryl\ Geller},
}

@book{ABB, place = {Cambridge}, series = {Cambridge Studies in Advanced Mathematics},
title = {A comprehensive introduction to sub-Riemannian geometry},
publisher = {Cambridge University Press},
author = {Agrachev, Andrei and Barilari, Davide and Boscain, Ugo}, year = {2019},
collection = {Cambridge Studies in Advanced Mathematics}}

@Inbook{Bellaiche_96,
author="Bella{\"i}che, Andr{\'e}",
title="The tangent space in sub-Riemannian geometry",
bookTitle="Sub-Riemannian Geometry",
year="1996",
publisher="Birkh{\"a}user Basel",
address="Basel",
pages="1--78",
isbn="978-3-0348-9210-0",
doi="10.1007/978-3-0348-9210-0_1",
url="https://doi.org/10.1007/978-3-0348-9210-0_1"
}

@article{Le_Donne_23,
	title = {Universal infinitesimal {H}ilbertianity of sub-{R}iemannian manifolds},
	volume = {59},
	issn = {1572-929X},
	url = {https://doi.org/10.1007/s11118-021-09971-8},
	doi = {10.1007/s11118-021-09971-8},
	number = {1},
	journal = {Potential Analysis},
	author = {Le Donne, Enrico and Lučić, Danka and Pasqualetto, Enrico},
	month = jun,
	year = {2023},
	pages = {349--374},
}

@misc{Magnabosco_Rossi_sub-Finsler_23,
	title = {Failure of the curvature-dimension condition in sub-{Finsler} manifolds},
	url = {http://arxiv.org/abs/2307.01820},
	doi = {10.48550/arXiv.2307.01820},
	language = {en},
	urldate = {2025-03-27},
	publisher = {arXiv},
	author = {Magnabosco, Mattia and Rossi, Tommaso},
	year = {2023},
	note = {arXiv:2307.01820},
}

@article{Juillet_20,
	title = {Sub-{Riemannian} structures do not satisfy {Riemannian} {Brunn}–{Minkowski} inequalities},
	volume = {37},
	issn = {0213-2230, 2235-0616},
	url = {https://ems.press/doi/10.4171/rmi/1205},
	doi = {10.4171/rmi/1205},
	language = {en},
	number = {1},
	urldate = {2025-03-24},
	journal = {Revista Matemática Iberoamericana},
	author = {Juillet, Nicolas},
	month = jul,
	year = {2020},
	pages = {177--188},
}

@article{Magnabosco_Rossi_AlmostRiem_23,
	title = {Almost-{Riemannian} manifolds do not satisfy the curvature-dimension condition},
	volume = {62},
	issn = {0944-2669, 1432-0835},
	url = {https://link.springer.com/10.1007/s00526-023-02466-x},
	doi = {10.1007/s00526-023-02466-x},
	language = {en},
	number = {4},
	urldate = {2025-03-24},
	journal = {Calculus of Variations and Partial Differential Equations},
	author = {Magnabosco, Mattia and Rossi, Tommaso},
	month = may,
	year = {2023},
	pages = {123},
}

@article{Magnabosco_Rossi_Review_2025,
	title = {A review of the tangent space in sub-{Finsler} geometry and applications to the failure of the {CD} condition},
	issn = {00193577},
	url = {https://doi.org/10.1016/j.indag.2025.05.001},
	doi = {https://doi.org/10.1016/j.indag.2025.05.001},
	language = {en},
        note = {https://doi.org/10.1016/j.indag.2025.05.001},
	urldate = {2025-06-11},
	journal = {Indagationes Mathematicae},
	author = {Magnabosco, Mattia and Rossi, Tommaso},
	month = may,
	year = {2025},
}

@article {Ambrosio_Stefani_2020,
    AUTHOR = {Ambrosio, Luigi and Stefani, Giorgio},
     TITLE = {Heat and entropy flows in {C}arnot groups},
   JOURNAL = {Rev. Mat. Iberoam.},
  FJOURNAL = {Revista Matem\'atica Iberoamericana},
    VOLUME = {36},
      YEAR = {2020},
    NUMBER = {1},
     PAGES = {257--290},
      ISSN = {0213-2230,2235-0616},
   MRCLASS = {53C17 (28A33 35K15 35R03 58J35)},
  MRNUMBER = {4061989},
MRREVIEWER = {Nathaniel\ Eldredge},
       DOI = {10.4171/rmi/1129},
       URL = {https://doi.org/10.4171/rmi/1129},
}

@article {Cavalletti_Mondino_2017,
    AUTHOR = {Cavalletti, Fabio and Mondino, Andrea},
     TITLE = {Sharp and rigid isoperimetric inequalities in metric-measure
              spaces with lower {R}icci curvature bounds},
   JOURNAL = {Invent. Math.},
  FJOURNAL = {Inventiones Mathematicae},
    VOLUME = {208},
      YEAR = {2017},
    NUMBER = {3},
     PAGES = {803--849},
      ISSN = {0020-9910,1432-1297},
   MRCLASS = {53C23 (49Q05 53C21)},
  MRNUMBER = {3648975},
MRREVIEWER = {Renjin\ Jiang},
       DOI = {10.1007/s00222-016-0700-6},
       URL = {https://doi.org/10.1007/s00222-016-0700-6},
}

@article {Lott_Villani_2009,
    AUTHOR = {Lott, John and Villani, C\'edric},
     TITLE = {Ricci curvature for metric-measure spaces via optimal
              transport},
   JOURNAL = {Ann. of Math. (2)},
  FJOURNAL = {Annals of Mathematics. Second Series},
    VOLUME = {169},
      YEAR = {2009},
    NUMBER = {3},
     PAGES = {903--991},
      ISSN = {0003-486X,1939-8980},
   MRCLASS = {53C23 (49Q15)},
  MRNUMBER = {2480619},
MRREVIEWER = {Alessio\ Figalli},
       DOI = {10.4007/annals.2009.169.903},
       URL = {https://doi.org/10.4007/annals.2009.169.903},
}

@article{Sturm_I_06,
issn = {0001-5962},
journal = {Acta {M}athematica},
pages = {65--131},
volume = {196},
publisher = {Kluwer Academic Publishers},
number = {1},
year = {2006},
title = {On the geometry of metric measure spaces. {I}},
language = {eng},
address = {Dordrecht},
author = {Sturm, Karl-Theodor},
}

@article{Sturm_II_06,
issn = {0001-5962},
journal = {Acta {M}athematica},
pages = {133--177},
volume = {196},
publisher = {Kluwer Academic Publishers},
number = {1},
year = {2006},
title = {On the geometry of metric measure spaces. {II}},
language = {eng},
address = {Dordrecht},
author = {Sturm, Karl-Theodor},
}

@misc{CJMRS_2025,
      title={Not all sub-{R}iemannian minimizing geodesics are smooth}, 
      author={Yacine Chitour and Frédéric Jean and Roberto Monti and Ludovic Rifford and Ludovic Sacchelli and Mario Sigalotti and Alessandro Socionovo},
      year={2025},
      note={arXiv: 2501.18920},
      archivePrefix={arXiv},
      primaryClass={math.DG},
      url={https://arxiv.org/abs/2501.18920}, 
}

@book {Montgomery_book,
    AUTHOR = {Montgomery, Richard},
     TITLE = {A tour of subriemannian geometries, their geodesics and
              applications},
    SERIES = {Mathematical Surveys and Monographs},
    VOLUME = {91},
 PUBLISHER = {American Mathematical Society, Providence, RI},
      YEAR = {2002},
     PAGES = {xx+259},
      ISBN = {0-8218-1391-9},
   MRCLASS = {53C17 (37J99 53C60 58E10 70G45 70H05)},
  MRNUMBER = {1867362},
MRREVIEWER = {Andrey\ V.\ Sarychev},
       DOI = {10.1090/surv/091},
       URL = {https://doi.org/10.1090/surv/091},
}

@article {Montgomery_1994,
    AUTHOR = {Montgomery, Richard},
     TITLE = {Abnormal minimizers},
   JOURNAL = {SIAM J. Control Optim.},
  FJOURNAL = {SIAM Journal on Control and Optimization},
    VOLUME = {32},
      YEAR = {1994},
    NUMBER = {6},
     PAGES = {1605--1620},
      ISSN = {0363-0129},
   MRCLASS = {49J15 (49K40 49Q20 58E25)},
  MRNUMBER = {1297101},
MRREVIEWER = {Constantin\ Udri\c ste},
       DOI = {10.1137/S0363012993244945},
       URL = {https://doi.org/10.1137/S0363012993244945},
}

@book{Gromov_98,
series = {Progress in Mathematics ; v. 152},
publisher = {Birkh\"auser},
isbn = {9780817638986},
year = {1998},
title = {Metric structures for Riemannian and non-Riemannian spaces: {W}ith Appendices by M. Katz, P. Pansu and S. Semmes},
language = {eng},
address = {Boston},
author = {Gromov, Mikhael},
lccn = {97024633},
}

@article {CC96,
    AUTHOR = {Cheeger, Jeff and Colding, Tobias H.},
     TITLE = {Lower bounds on {R}icci curvature and the almost rigidity of
              warped products},
   JOURNAL = {Ann. of Math. (2)},
  FJOURNAL = {Annals of Mathematics. Second Series},
    VOLUME = {144},
      YEAR = {1996},
    NUMBER = {1},
     PAGES = {189--237},
      ISSN = {0003-486X,1939-8980},
   MRCLASS = {53C21 (53C20 53C23)},
  MRNUMBER = {1405949},
MRREVIEWER = {Joseph\ E.\ Borzellino},
       DOI = {10.2307/2118589},
       URL = {https://doi.org/10.2307/2118589},
}

@Article{CCI,
	Author = {Cheeger, Jeff and Colding, Tobias H.},
	Title = {On the structure of spaces with {Ricci} curvature bounded below. {I}},
	FJournal = {Journal of Differential Geometry},
	Journal = {J. Differ. Geom.},
	ISSN = {0022-040X},
	Volume = {46},
	Number = {3},
	Pages = {406--480},
	Year = {1997},
	Language = {English},
	DOI = {10.4310/jdg/1214459974},
	zbMATH = {1145657},
	Zbl = {0902.53034}
}

@article{CCII,
issn = {0022-040X},
journal = {Journal of {D}ifferential {G}eometry},
pages = {13--35},
volume = {54},
publisher = {Lehigh University},
number = {1},
year = {2000},
title = {On the structure of spaces with {R}icci curvature bounded below. {II}},
language = {eng},
author = {Cheeger, Jeff and Colding, Tobias H.},
}

@article{CCIII,
issn = {0022-040X},
journal = {Journal of {D}ifferential Geometry},
pages = {37--74},
volume = {54},
publisher = {Lehigh University},
number = {1},
year = {2000},
title = {On the structure of spaces with {R}icci curvature bounded below. {III}},
language = {eng},
author = {Cheeger, Jeff and Colding, Tobias H.},
}

@article{Cordero-Erausquin_01,
	issn = {0020-9910},
	journal = {Inventiones Mathematicae},
	pages = {219--257},
	volume = {146},
	publisher = {Springer Nature},
	number = {2},
	year = {2001},
	title = {A {R}iemannian interpolation inequality à la {B}orell, {B}rascamp and {L}ieb},
	copyright = {Copyright 2005 Elsevier B.V., All rights reserved.},
	language = {eng},
	address = {NEW YORK},
	author = {Cordero-Erausquin, Dario and McCann, Robert J. and Schmuckenschläger, Michael},
}

@book{Villani_09,
series = {Grundlehren der mathematischen Wissenschaften ; 338},
publisher = {Springer},
isbn = {9783540710493},
year = {2009},
title = {Optimal transport : old and new},
language = {eng},
address = {Berlin},
author = {Villani, C{\'e}dric},
lccn = {2008932183},
}

@article{Ambrosio-Gigli-Savare_14,
AUTHOR = {Ambrosio, Luigi and Gigli, Nicola and Savar\'e, Giuseppe},
     TITLE = {Metric measure spaces with {R}iemannian {R}icci curvature
              bounded from below},
   JOURNAL = {Duke Math. J.},
  FJOURNAL = {Duke Mathematical Journal},
    VOLUME = {163},
      YEAR = {2014},
    NUMBER = {7},
     PAGES = {1405--1490},
      ISSN = {0012-7094,1547-7398},
   MRCLASS = {35R01 (60J45 60J65)},
  MRNUMBER = {3205729},
       DOI = {10.1215/00127094-2681605},
       URL = {https://doi.org/10.1215/00127094-2681605},
}

@article{Ohta_09,
    AUTHOR = {Ohta, Shin-ichi},
     TITLE = {Finsler interpolation inequalities},
   JOURNAL = {Calc. Var. Partial Differential Equations},
  FJOURNAL = {Calculus of Variations and Partial Differential Equations},
    VOLUME = {36},
      YEAR = {2009},
    NUMBER = {2},
     PAGES = {211--249},
      ISSN = {0944-2669,1432-0835},
   MRCLASS = {58E35 (53C60)},
  MRNUMBER = {2546027},
MRREVIEWER = {Giovanni\ Pisante},
       DOI = {10.1007/s00526-009-0227-4},
       URL = {https://doi.org/10.1007/s00526-009-0227-4},
}

@misc{Borza_Tashiro_24,
      title={Curvature-dimension condition of sub-{R}iemannian $\alpha$-{G}rushin half-spaces}, 
      author={Samuël Borza and Kenshiro Tashiro},
      year={2024},
      note={arXiv: 2409.11177},
      archivePrefix={arXiv},
      primaryClass={math.MG},
      url={https://arxiv.org/abs/2409.11177}, 
}

@misc{BMRT_24,
      title={Measure contraction property and curvature-dimension condition on sub-{F}insler {H}eisenberg groups}, 
      author={Samuël Borza and Mattia Magnabosco and Tommaso Rossi and Kenshiro Tashiro},
      year={2024},
      note={arXiv: 2402.14779},
      archivePrefix={arXiv},
      primaryClass={math.MG},
      url={https://arxiv.org/abs/2402.14779}, 
}

@article {Milman_21,
    AUTHOR = {Milman, Emanuel},
     TITLE = {The quasi curvature-dimension condition with applications to
              sub-{R}iemannian manifolds},
   JOURNAL = {Comm. Pure Appl. Math.},
  FJOURNAL = {Communications on Pure and Applied Mathematics},
    VOLUME = {74},
      YEAR = {2021},
    NUMBER = {12},
     PAGES = {2628--2674},
      ISSN = {0010-3640,1097-0312},
   MRCLASS = {53C17},
  MRNUMBER = {4373164},
MRREVIEWER = {Erlend\ Grong},
       DOI = {10.1002/cpa.21969},
       URL = {https://doi.org/10.1002/cpa.21969},
}

@article{Juillet_09,
    author = {Juillet, Nicolas},
    title = {Geometric Inequalities and Generalized {R}icci Bounds in the {H}eisenberg Group},
    journal = {International Mathematics Research Notices},
    volume = {2009},
    number = {13},
    pages = {2347-2373},
    year = {2009},
    month = {02},
    issn = {1073-7928},
    doi = {10.1093/imrn/rnp019},
    url = {https://doi.org/10.1093/imrn/rnp019},
    eprint = {https://academic.oup.com/imrn/article-pdf/2009/13/2347/1941806/rnp019.pdf},
}

@misc{Barilari_Mondino_Rizzi_22,
      title={Unified synthetic {R}icci curvature lower bounds for {R}iemannian and sub-{R}iemannian structures}, 
      author={Davide Barilari and Andrea Mondino and Luca Rizzi},
      year={2024},
      note={arXiv: 2211.07762 (To appear in Memoirs of the AMS)},
      archivePrefix={arXiv},
      primaryClass={math.DG},
      url={https://arxiv.org/abs/2211.07762}, 
}

@misc{Borza_Rizzi_25,
      title={Failure of the measure contraction property via quotients in higher-step sub-{R}iemannian structures}, 
      author={Samuël Borza and Luca Rizzi},
      year={2025},
      note={arXiv: 2505.09681},
      archivePrefix={arXiv},
      primaryClass={math.DG},
      url={https://arxiv.org/abs/2505.09681}, 
}

@article {Rifford_13,
    AUTHOR = {Rifford, Ludovic},
     TITLE = {Ricci curvatures in {C}arnot groups},
   JOURNAL = {Math. Control Relat. Fields},
  FJOURNAL = {Mathematical Control and Related Fields},
    VOLUME = {3},
      YEAR = {2013},
    NUMBER = {4},
     PAGES = {467--487},
      ISSN = {2156-8472,2156-8499},
   MRCLASS = {53C23 (22E25 58C15 58J60)},
  MRNUMBER = {3110060},
MRREVIEWER = {Igor\ Belegradek},
       DOI = {10.3934/mcrf.2013.3.467},
       URL = {https://doi.org/10.3934/mcrf.2013.3.467},
}

@article {Rizzi_16,
    AUTHOR = {Rizzi, Luca},
     TITLE = {Measure contraction properties of {C}arnot groups},
   JOURNAL = {Calc. Var. Partial Differential Equations},
  FJOURNAL = {Calculus of Variations and Partial Differential Equations},
    VOLUME = {55},
      YEAR = {2016},
    NUMBER = {3},
     PAGES = {Art. 60, 20},
      ISSN = {0944-2669,1432-0835},
   MRCLASS = {53C17 (35R03 53C21 53C22 53C23 54E35)},
  MRNUMBER = {3502622},
MRREVIEWER = {Enrico\ Le Donne},
       DOI = {10.1007/s00526-016-1002-y},
       URL = {https://doi.org/10.1007/s00526-016-1002-y},
}

@article {Barilari_Rizzi_18,
    AUTHOR = {Barilari, Davide and Rizzi, Luca},
     TITLE = {Sharp measure contraction property for generalized {H}-type
              {C}arnot groups},
   JOURNAL = {Commun. Contemp. Math.},
  FJOURNAL = {Communications in Contemporary Mathematics},
    VOLUME = {20},
      YEAR = {2018},
    NUMBER = {6},
     PAGES = {1750081, 24},
      ISSN = {0219-1997,1793-6683},
   MRCLASS = {53C17 (35R03 53C21 53C22 54E35)},
  MRNUMBER = {3848070},
MRREVIEWER = {Nicolas\ Juillet},
       DOI = {10.1142/S021919971750081X},
       URL = {https://doi.org/10.1142/S021919971750081X},
}

@article {Badreddine_Rifford_20,
    AUTHOR = {Badreddine, Zeinab and Rifford, Ludovic},
     TITLE = {Measure contraction properties for two-step analytic
              sub-{R}iemannian structures and {L}ipschitz {C}arnot groups},
   JOURNAL = {Ann. Inst. Fourier},
  FJOURNAL = {Universit\'e{} de Grenoble. Annales de l'Institut Fourier},
    VOLUME = {70},
      YEAR = {2020},
    NUMBER = {6},
     PAGES = {2303--2330},
      ISSN = {0373-0956,1777-5310},
   MRCLASS = {53C17 (22E25 53C23 58C15 58J60)},
  MRNUMBER = {4245620},
       DOI = {10.5802/aif.3362},
       URL = {https://doi.org/10.5802/aif.3362},
}

@article {Barilari_Rizzi_19,
    AUTHOR = {Barilari, Davide and Rizzi, Luca},
     TITLE = {Sub-{R}iemannian interpolation inequalities},
   JOURNAL = {Invent. Math.},
  FJOURNAL = {Inventiones Mathematicae},
    VOLUME = {215},
      YEAR = {2019},
    NUMBER = {3},
     PAGES = {977--1038},
      ISSN = {0020-9910,1432-1297},
   MRCLASS = {53C17 (49J15 49Q20)},
  MRNUMBER = {3935035},
MRREVIEWER = {Emmanuel\ Tr\'elat},
       DOI = {10.1007/s00222-018-0840-y},
       URL = {https://doi.org/10.1007/s00222-018-0840-y},
}

@article {Barilari_Rizzi_20,
    AUTHOR = {Barilari, Davide and Rizzi, Luca},
     TITLE = {Bakry-\'{E}mery curvature and model spaces in sub-{R}iemannian
              geometry},
   JOURNAL = {Math. Ann.},
  FJOURNAL = {Mathematische Annalen},
    VOLUME = {377},
      YEAR = {2020},
    NUMBER = {1-2},
     PAGES = {435--482},
      ISSN = {0025-5831,1432-1807},
   MRCLASS = {53C17 (49J15 49N10)},
  MRNUMBER = {4099638},
MRREVIEWER = {Jing\ Wang},
       DOI = {10.1007/s00208-020-01982-x},
       URL = {https://doi.org/10.1007/s00208-020-01982-x},
}

@article {A_LD_NG_23,
    AUTHOR = {Antonelli, Gioacchino and Le Donne, Enrico and Nicolussi Golo,
              Sebastiano},
     TITLE = {Lipschitz {C}arnot-{C}arath\'eodory structures and their
              limits},
   JOURNAL = {J. Dyn. Control Syst.},
  FJOURNAL = {Journal of Dynamical and Control Systems},
    VOLUME = {29},
      YEAR = {2023},
    NUMBER = {3},
     PAGES = {805--854},
      ISSN = {1079-2724,1573-8698},
   MRCLASS = {53C17 (22E60 28A75 49J52 49Q15 53C60)},
  MRNUMBER = {4645068},
MRREVIEWER = {Davide\ Vittone},
       DOI = {10.1007/s10883-022-09613-1},
       URL = {https://doi.org/10.1007/s10883-022-09613-1},
}

@book {LeD_book,
    AUTHOR = {Le Donne, Enrico},
     TITLE = {Metric {L}ie groups---{C}arnot-{C}arath\'eodory spaces from
              the homogeneous viewpoint},
    SERIES = {Graduate Texts in Mathematics},
    VOLUME = {306},
 PUBLISHER = {Springer, Cham},
      YEAR = {[2025] \copyright 2025},
     PAGES = {xvi+480},
      ISBN = {978-3-031-98831-8; 978-3-031-98832-5},
   MRCLASS = {99-06},
  MRNUMBER = {4995404},
       DOI = {10.1007/978-3-031-98832-5},
       URL = {https://doi.org/10.1007/978-3-031-98832-5},
}

@article{aradi_isometries_2014,
	title = {Isometries, submetries and distance coordinates on {Finsler} manifolds},
	volume = {143},
	copyright = {http://www.springer.com/tdm},
	issn = {0236-5294, 1588-2632},
	url = {http://link.springer.com/10.1007/s10474-013-0381-1},
	doi = {10.1007/s10474-013-0381-1},
	language = {en},
	number = {2},
	urldate = {2025-06-24},
	journal = {Acta Mathematica Hungarica},
	author = {Aradi, Bernadett and Kertész, Dávid Csaba},
	month = aug,
	year = {2014},
	pages = {337--350},
}

@book{jean_control_2014,
	address = {Cham},
	series = {{SpringerBriefs} in {Mathematics}},
	title = {Control of nonholonomic systems: From sub-{R}iemannian geometry to motion planning},
	copyright = {https://www.springernature.com/gp/researchers/text-and-data-mining},
	isbn = {978-3-319-08689-7 978-3-319-08690-3},
	shorttitle = {Control of {Nonholonomic} {Systems}},
	url = {https://link.springer.com/10.1007/978-3-319-08690-3},
	language = {en},
	urldate = {2025-06-23},
	publisher = {Springer International Publishing},
	author = {Jean, Frédéric},
	year = {2014},
	doi = {10.1007/978-3-319-08690-3},
}

@article{gigli_euclidean_2015,
	title = {Euclidean spaces as weak tangents of infinitesimally {Hilbertian} metric measure spaces with {Ricci} curvature bounded below},
	volume = {2015},
	url = {https://doi.org/10.1515/crelle-2013-0052},
	doi = {doi:10.1515/crelle-2013-0052},
	number = {705},
	urldate = {2025-06-25},
	journal = {Journal für die reine und angewandte Mathematik (Crelles Journal)},
	author = {Gigli, Nicola and Mondino, Andrea and Rajala, Tapio},
	year = {2015},
	pages = {233--244},
}

@article {BKS18,
    AUTHOR = {Balogh, Zolt\'an M. and Krist\'aly, Alexandru and Sipos,
              Kinga},
     TITLE = {Geometric inequalities on {H}eisenberg groups},
   JOURNAL = {Calc. Var. Partial Differential Equations},
  FJOURNAL = {Calculus of Variations and Partial Differential Equations},
    VOLUME = {57},
      YEAR = {2018},
    NUMBER = {2},
     PAGES = {Paper No. 61, 41},
      ISSN = {0944-2669,1432-0835},
   MRCLASS = {49Q20 (53C17)},
  MRNUMBER = {3774461},
MRREVIEWER = {Alireza\ Ranjbar-Motlagh},
       DOI = {10.1007/s00526-018-1320-3},
       URL = {https://doi.org/10.1007/s00526-018-1320-3},
}

@article {BMRT25,
    AUTHOR = {Borza, Samu\"el and Magnabosco, Mattia and Rossi, Tommaso and
              Tashiro, Kenshiro},
     TITLE = {Curvature exponent of sub-{F}insler {H}eisenberg groups},
   JOURNAL = {SIAM J. Math. Anal.},
  FJOURNAL = {SIAM Journal on Mathematical Analysis},
    VOLUME = {57},
      YEAR = {2025},
    NUMBER = {4},
     PAGES = {3561--3586},
      ISSN = {0036-1410,1095-7154},
   MRCLASS = {53C17 (49Q22 53C23 53C60)},
  MRNUMBER = {4927105},
       DOI = {10.1137/24M1690692},
       URL = {https://doi.org/10.1137/24M1690692},
}

@article {BG17,
    AUTHOR = {Baudoin, Fabrice and Garofalo, Nicola},
     TITLE = {Curvature-dimension inequalities and {R}icci lower bounds for
              sub-{R}iemannian manifolds with transverse symmetries},
   JOURNAL = {J. Eur. Math. Soc. (JEMS)},
  FJOURNAL = {Journal of the European Mathematical Society (JEMS)},
    VOLUME = {19},
      YEAR = {2017},
    NUMBER = {1},
     PAGES = {151--219},
      ISSN = {1435-9855,1435-9863},
   MRCLASS = {53C17 (35H10 53C12 53C21 58J60 58J65)},
  MRNUMBER = {3584561},
MRREVIEWER = {Nathaniel\ Eldredge},
       DOI = {10.4171/JEMS/663},
       URL = {https://doi.org/10.4171/JEMS/663},
}

@article {CG_split,
    AUTHOR = {Cheeger, Jeff and Gromoll, Detlef},
     TITLE = {The splitting theorem for manifolds of nonnegative {R}icci
              curvature},
   JOURNAL = {J. Differential Geometry},
  FJOURNAL = {Journal of Differential Geometry},
    VOLUME = {6},
      YEAR = {1971/72},
     PAGES = {119--128},
      ISSN = {0022-040X,1945-743X},
   MRCLASS = {53C20},
  MRNUMBER = {303460},
MRREVIEWER = {J.\ R.\ Vanstone},
       URL = {http://projecteuclid.org/euclid.jdg/1214430220},
}

@book {Pet_book,
    AUTHOR = {Petersen, Peter},
     TITLE = {Riemannian geometry},
    SERIES = {Graduate Texts in Mathematics},
    VOLUME = {171},
   EDITION = {Third},
 PUBLISHER = {Springer, Cham},
      YEAR = {2016},
     PAGES = {xviii+499},
      ISBN = {978-3-319-26652-7; 978-3-319-26654-1},
   MRCLASS = {53-01 (53C20 53C21 53C23)},
  MRNUMBER = {3469435},
       DOI = {10.1007/978-3-319-26654-1},
       URL = {https://doi.org/10.1007/978-3-319-26654-1},
}

@article{AGS_convergence,
author = {Gigli, Nicola and Mondino, Andrea and Savaré, Giuseppe},
title = {Convergence of pointed non-compact metric measure spaces and stability of Ricci curvature bounds and heat flows},
journal = {Proceedings of the London Mathematical Society},
volume = {111},
number = {5},
pages = {1071-1129},
doi = {https://doi.org/10.1112/plms/pdv047},
url = {https://londmathsoc.onlinelibrary.wiley.com/doi/abs/10.1112/plms/pdv047},
eprint = {https://londmathsoc.onlinelibrary.wiley.com/doi/pdf/10.1112/plms/pdv047},
year = {2015}
}

\end{document}